\documentclass[12pt]{amsart}

\usepackage{graphicx}
\usepackage{mathtools}
\usepackage{fullpage}
\usepackage{amsmath}
\usepackage{amsthm}
\usepackage{amssymb}
\usepackage[numeric]{amsrefs}
\usepackage{mathrsfs}
\usepackage{color}
\usepackage{tikz}
\usepackage{hyperref}
\usepackage{stmaryrd}
\usetikzlibrary{arrows}
\usetikzlibrary{positioning}
\usetikzlibrary{cd}

\newcommand{\lfrak}{\mathfrak{l}}

\newcommand{\pfrak}{\mathfrak{p}}

\newcommand{\sfrak}{\mathfrak{s}}

\newcommand{\Afrak}{\mathfrak{A}}

\newcommand{\Gfrak}{\mathfrak{G}}
\newcommand{\Hfrak}{\mathfrak{H}}

\newcommand{\Lfrak}{\mathfrak{L}}

\newcommand{\cbf}{\mathbf{c}}

\newcommand{\wbf}{\mathbf{w}}
\newcommand{\xbf}{\mathbf{x}}

\newcommand{\zbf}{\mathbf{z}}

\newcommand{\Dbf}{\mathbf{D}}

\newcommand{\Hbf}{\mathbf{H}}
\newcommand{\Ibf}{\mathbf{I}}

\newcommand{\Lbf}{\mathbf{L}}

\newcommand{\Ubf}{\mathbf{U}}

\newcommand{\Ccal}{\mathcal{C}}

\newcommand{\Mcal}{\mathcal{M}}

\newcommand{\Ocal}{\mathcal{O}}

\newcommand{\Rcal}{\mathcal{R}}

\newcommand{\Tcal}{\mathcal{T}}
\newcommand{\Ucal}{\mathcal{U}}
\newcommand{\Vcal}{\mathcal{V}}

\newcommand{\Hscr}{\mathscr{H}}

\newcommand{\Lscr}{\mathscr{L}}

\newcommand{\Abb}{\mathbb{A}}

\newcommand{\Gbb}{\mathbb{G}}

\newcommand{\Pbb}{\mathbb{P}}
\newcommand{\Qbb}{\mathbb{Q}}

\newcommand{\Zbb}{\mathbb{Z}}

\newcommand{\Var}{\mathbf{Var}}

\newcommand{\Gal}{\operatorname{Gal}}

\newcommand{\Spec}{\operatorname{Spec}}

\newcommand{\Zhat}{\widehat{\Zbb}}
\newcommand{\Out}{\operatorname{Out}}
\newcommand{\Inn}{\operatorname{Inn}}
\newcommand{\cont}{\mathrm{cont}}
\newcommand{\smin}{\smallsetminus}
\newcommand{\Qbar}{\overline{\Qbb}}
\newcommand{\HH}{\operatorname{H}}

\newcommand{\GSL}{{\Gamma\mathrm{L}}}
\newcommand{\PGSL}{{\mathrm{P}\Gamma\mathrm{L}}}

\newcommand{\trdeg}{\operatorname{trdeg}}

\newcommand{\Aut}{\operatorname{Aut}}
\newcommand{\PAut}{\operatorname{PAut}}

\newcommand{\sep}{\mathrm{sep}}

\newcommand{\Hom}{\operatorname{Hom}}

\newcommand{\GT}{\mathrm{GT}}

\newcommand{\GThat}{\widehat{\GT}}

\newcommand{\dec}{\mathrm{dec}}

\newcommand{\Pibf}{\boldsymbol{\Pi}}
\newcommand{\pibf}{\boldsymbol{\pi}}
\newcommand{\HHbf}{\boldsymbol{\HH}}

\newcommand{\Ccalbf}{\boldsymbol{\Ccal}}
\newcommand{\iotabf}{\boldsymbol{\iota}}

\newcommand{\Planes}{\mathbf{Planes}}

\newcommand{\image}{\operatorname{image}}

\newcommand{\et}{{\text{\'et}}}
\newcommand{\Profout}{\mathbf{Prof}_{\Out}}
\newcommand{\bir}{\mathrm{bir}}

\newtheorem{theorem}{Theorem}
\newtheorem{lemma}[theorem]{Lemma}
\newtheorem{proposition}[theorem]{Proposition}
\newtheorem{fact}[theorem]{Fact}

\newtheorem{maintheorem}{Theorem}

\newtheorem*{maintheorem*}{Main Theorem}

\theoremstyle{definition}
\newtheorem{condition}{Condition}
\numberwithin{condition}{subsection}

\theoremstyle{remark}
\newtheorem{remark}[theorem]{Remark}

\renewcommand{\hat}{\widehat}
\renewcommand{\bar}{\overline}
\renewcommand{\vec}{\overrightarrow}
\renewcommand{\tilde}{\widetilde}
\renewcommand{\epsilon}{\varepsilon}

\numberwithin{equation}{section}
\numberwithin{theorem}{section}
\title{Line and Hyperplane GT-variants}
\date{\today}

\author{Florian Pop}
\thanks{F. Pop was supported by Simons Collaboration Grant 712449.}
\email{pop[at]math[dot]upenn[dot]edu}
\address{University of Pennsylvania,
Department of Mathematics,
209 S 33rd St, Philadelphia, PA 19104, USA}

\author{Adam Topaz}
\thanks{A. Topaz was supported by an NSERC discovery grant and the University of Alberta.}
\email{topaz[at]ualberta[dot]ca}
\address{Mathematical and Statistical Sciences,
University of Alberta,
632 Central Academic Building,
Edmonton, Alberta T6G 2G1,
Canada}

  

\begin{document}
\begin{abstract}
  In this work, we introduce a variant of the Grothendieck-Teichm{\"u}ller group, defined in terms of complements of hyperplane arrangements and pro-$\ell$ two-step nilpotent fundamental groups, and prove that it is isomorphic to the absolute Galois group of $\Qbb$.
\end{abstract}

\maketitle

\setcounter{tocdepth}{1}
\tableofcontents

\section{Introduction}

One of the primary themes in Grothendieck's \emph{Esquisse d'un Programme} was to study the structure of the absolute Galois group of $\Qbb$ (and other fields) via its action on geometric objects, specifically (geometric \'etale) fundamental groups of algebraic varieties.
In this paper, we study this Galois action on certain natural quotients of the geometric \'etale fundamental groups of complements of hyperplane arrangements.
Our main result shows that the absolute Galois group of $\Qbb$ itself can be \emph{determined entirely} in terms of the (outer) automorphisms of such fundamental groups.

The primary motivation for this work arises from the relationship between the absolute Galois group of $\Qbb$ and the Grothendieck-Teichm{\"u}ller group $\GThat$.
Drinfel'd~\cite{Drinfeld1990oqqhaaoagticcwgq} described $\GThat$ \emph{explicitly} as a subgroup of $\Aut(\hat F_{2})$, where $\hat F_{2}$ denotes the free profinite group on two letters, say $x,y$.
Concretely, $\GThat$ consists of elements $\sigma \in \Aut(\hat F_{2})$ which act on $x,y$ as $\sigma(x) = x^{\lambda}$, $\sigma(y)=f^{-1}y^{\lambda} f$ for some $\lambda \in \Zhat^{\times}$ and $f$ in the (closed) commutator subgroup of $\hat F_{2}$, where the pair $(\lambda,f)$ satisfies three \emph{explicit} relations.
Essentially as a consequence of Belyi's theorem~\cite{Belyi}, it is known that that $\Gal(\Qbar|\Qbb)$ has a canonical embedding into $\Aut(\hat F_{2})$, via the identification of $\hat F_{2}$ with $\pi_{1}^{\et}(\Pbb^{1} \smin \{0,1,\infty\}, \vec{01})$.
The image of this embedding is contained in $\GThat$~\cite{Belyi}.
Arguably the most important open problem in this area is the so-called \emph{Grothendieck-Teichm\"uller conjecture} which predicts that the canonical map
\[ \Gal(\Qbar|\Qbb) \hookrightarrow \GThat \]
discussed above is an \emph{isomorphism}.

The work around $\GThat$ is quite extensive, originating with Drinfel'd~\cite{Drinfeld1990oqqhaaoagticcwgq}, Ihara~\cites{Ihara1986, Ihara1991, Ihara2000}, Deligne~\cite{Deligne}, followed by many others, e.g.~Hain-Matsumoto~\cite{HainMatsumoto}, Matsumoto~\cite{Matsumoto1996}, Harbater-Schneps~\cite{HarbaterSchneps2000fgomatgtg}, Ihara-Matsumoto~\cite{IharaMatsumoto}, Lochak-Schneps~\cite{LochakSchneps}, Nakamura-Schneps~\cite{NakamuraSchneps}, Hoshi-Mochizuki~\cite{Hoshi-Mochizuki2011}, Enriquez~\cite{Enriquez}, and including some very recent work due to Hoshi-Minamide-Mochizuki~\cite{HoshiMinamideMochizuki} and Minamide-Nakamura~\cite{MinamideNakamura2019tagotpbg}.
Nevertheless, the precise relationship between $\Gal(\Qbar|\Qbb)$ and $\GThat$, particularly whether the map $\Gal(\Qbar|\Qbb) \to \GThat$ is \emph{surjective}, remains completely open to this day.

Let us now take a more functorial point of view.
Write $\Profout$ for the category whose objects are \emph{profinite groups} and whose morphisms are continuous \emph{outer}-homomorphisms.
In other words,
\[ \Hom_{\Profout}(G,H) = \Hom_{\cont}(G,H)/\Inn(H), \]
where $\Hom_{\cont}(G,H)$ is the set of continuous homomorphisms $G \to H$ and the inner automorphism group $\Inn(H)$ acts by postcomposition.

For categories $\Vcal$ of geometrically integral $\Qbb$-varieties, let 
\[
\bar\pi_{\Vcal}: \Vcal \,\longrightarrow \, \Profout, \quad X \mapsto \pi_{1}^{\et}(X \otimes_{\Qbb}\Qbar)
\]
denote the geometric fundamental group functor.
Since we consider this functor as taking values in $\Profout$, the choice of basepoints in the computation of the fundamental group is irrelevant, and is therefore omitted from the notation above.
One has a canonical representation $\rho_{\Vcal} : \Gal(\Qbar|\Qbb) \to \Aut(\bar\pi_{\Vcal})$, and Grothendieck suggested studying $\Gal(\Qbar|\Qbb)$ via $\rho_{\Vcal}$ for categories $\Vcal$, such as the full Teichm\"uller modular tower $\Tcal={\{\Mcal_{g,n}\}}_{g,n}$, see~\cite{zbMATH01088533}.
If $\Vcal$ contains a hyperbolic curve, e.g.~$\Mcal_{0,4}= \Pbb^{1} \smin \{0,1,\infty\}$, then $\rho_{\Vcal}$ is known to be injective by work by Drinfel'd~\cite{Drinfeld1990oqqhaaoagticcwgq}, Voevodski~\cite{Voevodski1991}, Matsumoto~\cite{Matsumoto1996} and Hoshi-Mochizuki~\cite{Hoshi-Mochizuki2011}.

In relation to $\GThat$, Harbater-Schneps~\cite{HarbaterSchneps2000fgomatgtg} showed that $\GThat=\Aut^{*}(\bar\pi_{\Vcal_0})$, where $\Vcal_0$ is the full subcategory of $\Tcal$ whose objects are $\Mcal_{0,4}$ and $\Mcal_{0,5}$, while $\Aut^{*}$ refers to the collection automorphisms $\sigma \in \Aut(\bar\pi_{\Vcal_{0}})$ which preserve the conjugacy classes of ``inertia at infinity.''
On the other hand, there is recent quite significant progress on the relationship between $\GThat$ and $\Aut(\bar\pi_{\Vcal})$ with $\Vcal \subset \Tcal$, arising from the work of Hoshi-Minamide-Mochizuki~\cite{HoshiMinamideMochizuki} and Minamide-Nakamura~\cite{MinamideNakamura2019tagotpbg}.
Among other things, it follows from these works that
\[ \GThat = \Aut(\bar\pi_{\Vcal_0}) = \Aut(\bar\pi_{\Tcal_0}), \]
where $\Tcal_0$ is the genus zero part of $\Tcal$, while also $\GThat \cong \Out\left(\pi_{1}^{\et}(\Mcal_{1,2} \otimes_{\Qbb} \Qbar)\right)$.
As mentioned above, $\rho : \Gal(\Qbar|\Qbb) \to \GThat$ is known to be injective since $\Mcal_{0,4}$ is contained in $\Vcal_{0}$.
However, the \emph{surjectivity} of this map, or equivalently, the question of whether this map is an isomorphism, is one of the most important open problems in modern Galois theory.
This question is an active area of research, and has been studied by many authors. 

\subsection{Connections with the I/OM}\label{subsec:intro-IOM}
The issue of surjectivity mentioned above is related to a conjecture due to Ihara/Oda-Matsumoto, the \emph{classical I/OM} for short, asserting that $\rho_{\Vcal}$ is an isomorphism in the case where $\Vcal$ is the category of \emph{all geometrically integral $\Qbb$-varieties}.
The classical I/OM has a positive answer, see~Pop~\cite{Pop2019ftvoiomoiqomc}, Introduction, for a short historical note.
However, the solution to the classical I/OM (and its refinements/extensions) is obtained via completing \emph{Bogomolov's Program in birational anabelian geometry} (BP)~\cite{Bogomolov1991} under additional hypotheses which are satisfied in the context of the I/OM;~see \cites{Pop2019ftvoiomoiqomc,Topaz2018tgaoglatmliom}.

More precisely, BP, which is still essentially open in general, asserts that a function field $K|k$, with $\trdeg(K|k)>1$ and $k$ algebraically closed, can be recovered, in a suitable functorial sense, from its pro-$\ell$ abelian-by-central Galois group $\Pi^{c}_{K}\to \Pi^{a}_{K}$ (see \S\ref{subsec:intro-fund-groups} below for notations).
In this context we consider the group
\[ \PAut^{c}(\Pi_{K}^{a}) := \image(\Aut(\Pi_{K}^{c}) \to \Aut(\Pi_{K}^{a}))/\Zbb_{\ell}^{\times} \]
as well as the canonical map $\Aut(K) \to \PAut^{c}(\Pi_{K}^{a})$.
We say that ``BP holds'' provided that $\Aut(K) \to \PAut^{c}(\Pi_{K}^{a})$ is an isomorphism.
If BP also holds for all finite extensions $L|K$, then it follows that the canonical map $\Aut(K) \to \Out(\Gal(K^{\sep}|K))$ is also an isomorphism.

This approach can be used to study $\Gal(\Qbar|\Qbb)$ in the case $\Vcal = \Vcal_{0}^{\bir}$, a \emph{birational} variant of $\Vcal_{0}$ mentioned above, as follows.
First, recall that
\[ \Mcal_{0,4} = \Pbb^{1} \smin \{0,1,\infty\}, \ \ \Mcal_{0,5} = (\Mcal_{0,4} \times \Mcal_{0,4}) \smin \Delta \]
were $\Delta$ is the diagonal.
In particular, $\Mcal_{0,5}$ can be identified with the open affine subvariety $\Abb^{2} \smin \Lbf_{0}$ of $\Abb^{2} = \Spec \Qbb[x,y]$, where $\Lbf_{0}$ is the zero-locus of the following function:
\[ x \cdot (1-x) \cdot y \cdot (1-y) \cdot (x-y). \]
The objects of $\Vcal_{0}^{\bir}$ consist of $\Mcal_{0,4}$, $\Mcal_{0,5}$, and any nonempty open affine $\Qbb$-subvariety $U$ of $\Mcal_{0,5}$, while the morphisms of $\Vcal_{0}^{\bir}$ are the identity morphisms, the inclusions among the $U$, and the projections
\[ \pi_t : U \to \Pbb^{1} \smin \{0,1,\infty\} = \Mcal_{0,4} \]
defined by any one of the functions $t \in \Sigma_{0} := \{x,y,x-y\}$ whenever $U$ is disjoint from the base locus of the rational map $t$.

Setting $K_{0} = \Qbb(x,y)$, $k = \Qbar$, and $K = k(x,y)$, the morphisms $\pi_{t} : \Spec K_{0} \to \Mcal_{0,4}$ for $t \in \Sigma_{0}$ give rise to projections $\pi_{t} : \Gal(\bar K|K) \to \pi_1^{\et}(\Mcal_{0,4} \otimes_{\Qbb} \Qbar)$ in $\Profout$.
Let $\Aut_{\Sigma_{0}}(K) \subset \Aut(K)$ denote the subgroup of automorphisms which preserve the subring $k[t,1/t,1/(1-t)]$ for every $t \in \Sigma_{0}$, and $\Out_{\Sigma_0}(\Gal(\bar K|K)) \subset \Out(\Gal(\bar K|K))$ be the subgroup of all the automorphisms which preserve the kernels of  the projections $\pi_{t}$ for $t\in\Sigma_0$.
One obtains an embedding
\[ \Aut(\bar\pi_{\Vcal_{0}^{\bir}}) \hookrightarrow \Out_{\Sigma_{0}}(\Gal(\bar K|K)) \]
by taking limits along the various open subvarieties $U$.
It turns out that endowing $\Gal(\bar K|K)$ with the additional data of the projections $\pi_{t}$ for $t \in \Sigma_{0}$ is sufficient to complete BP for $K$ and its finite extensions.
Using this, one deduces that $\Aut_{\Sigma_{0}}(K) = \Out_{\Sigma_{0}}(\Gal(\bar K|K))$, while the $\pi_{t}$, $t \in \Sigma_{0}$, rigidify the situation so that one has
\[ \Gal(\Qbar|\Qbb) \hookrightarrow \Aut(\bar\pi_{\Vcal_{0}^{\bir}}) \hookrightarrow \Out_{\Sigma_0}(\Gal(\bar K|K))= \Gal(\Qbar|\Qbb). \]
This finally shows that the canonical map $\Gal(\Qbar|\Qbb) \to \Aut(\bar\pi_{\Vcal_{0}^{\bir}})$ is an isomorphism.
Replacing $\pi_1^{\et}$ by its pro-$\ell$ abelian-by-central quotient $\Pi^c \to \Pi^a$, and $\Out(\bar\pi_{\Vcal})$ with $\PAut^c(\Pi^{a}_{\Vcal})$ (again, see \S\ref{subsec:intro-fund-groups} for the notation), the analogous assertions hold in this setting as well, hence the canonical map:
\[ \Gal(\Qbar|\Qbb) \to \PAut^{c}(\Pi^{a}_{\Vcal_{0}^{\bir}}) \]
is an \emph{isomorphism}.
See~\cite{Pop2019ftvoiomoiqomc} for more details.

\subsection{The category $\Lscr_0$}\label{subsec:Lscr}
In this work, we consider the \emph{line-arrangement} variant of $\Vcal_0^{\bir}$, and prove similar results in this new context which is much more restrictive in the Galois-theoretic sense.
Precisely, let $\Lscr_{0}\subset\Vcal_0^{\bir}$ be the category of $\Qbb$-varieties whose objects are $\Mcal_{0,4}$, $\Mcal_{0,5}$ and all $\Ucal_\Lbf = \Abb^2 \smin \Lbf$ where $\Lbf$ is a closed $\Qbb$-subvariety which is (geometrically) a finite union of affine lines, and whose morphisms are the identity morphisms, the inclusions $\Ucal_{\Lbf_{1}} \hookrightarrow \Ucal_{\Lbf_{2}}$ for $\Lbf_{2} \subset \Lbf_{1}$, and the projections $\pi_{t} : \Ucal_{\Lbf} \to \Pbb^{1} \smin \{0,1,\infty\}$ for $t \in \Sigma_{0}$ with $\Sigma_{0}$ as above, whenever $\Lbf$ is sufficiently large.
A main consequence of our general result (see Theorem~\ref{maintheorem:general}) is as follows.
\begin{maintheorem*}\label{maintheorem:intro}
The canonical map $\rho : \Gal(\Qbar|\Qbb) \to \PAut^{c}(\Pi^{a}_{\Lscr_{0}})$ is an isomorphism.
\end{maintheorem*}

Since $\Lscr_{0} \subset \Vcal_{0}^{\bir}$, the Main Theorem above could be seen as an intermediate step between the category $\Vcal_0^{\bir}$ which is of \emph{birational nature} and yields $\Gal(\Qbar|\Qbb)$, and the category $\Vcal_{0}$ which is {\it not of birational nature\/} and yields $\GThat$. Therefore we view this work as a % can therefore be seen as a significant
step toward understanding the relationship between $\Gal(\Qbar|\Qbb)$ and $\GThat$, because there is no apparent birational content to the categories of varieties used in defining the latter.

\subsection{Strategy of proof}\label{subsection:sketch}
The general outline of the proof of the Main Theorem above is as follows.
See also the discussion in the next subsection for a comparison of
the strategy and techniques of this paper with the ones used to prove that $\Gal(\Qbar|\Qbb)=\PAut^c(\Pi^a_{\Vcal_{0}^{\bir}})$.

Most of the work takes place inside of the following limit object:
\[ \Pibf^{\star} := \varprojlim_{\Lbf} \Pi_{\Ucal_{\Lbf}}^{\star},\quad \star \in \{a,c\}, \]
where $\Lbf \subset \Abb^{2}_{\Qbar}$ varies over all line arrangements defined over $\Qbb$.
The argument proceeds roughly as follows:
\begin{enumerate}
  \item First, we recover the inertia groups associated to all lines inside of $\Pibf^{a}$ using $\Pibf^{c}$ along with some additional data arising from the structure of $\Lscr_{0}$, thereby obtaining an action of $\Aut^{c}(\Pi^{a}_{\Lscr_{0}})$ on the collection of lines.
  \item Second, we identify the collection of lines with points in the dual projective plane, and show that this action is compatible with the lines in this projective space.
  \item Finally, we apply the fundamental theorem of projective geometry~\cite{artin} and eventually show that this action factors through $\Gal(\Qbar|\Qbb)$.
  \item To conclude, we prove that the kernel of the induced map $\Aut^{c}(\Pi^{a}_{\Lscr_{0}}) \to \Gal(\Qbar|\Qbb)$ is contained in the image of the canonical map $\Zbb_{\ell}^{\times} \to \Aut(\Pi^{a}_{\Lscr_{0}})$.
\end{enumerate}

\subsection{Comparison with the proof of $\Gal(\Qbar|\Qbb)= \PAut^c(\Pi^a_{\Vcal_{0}^\bir})$}
The description of $\Gal(\Qbar|\Qbb)$ arising from BP, proceeds roughtl as follows.
For $K_{0} = \Qbb(x,y)$, $k = \Qbar$ and $K = k(x,y)$ as in \S\ref{subsec:intro-IOM},
let $\hat{K^{\times}}$ denote the $\ell$-adic completion of the multiplicative group $K^{\times}$.
Using Kummer theory and a fixed isomorphism $\Zbb_{\ell}(1) \cong \Zbb_{\ell}$, one has an identification:
\[ \hat{K^{\times}} \, \cong \, \Hom(\Pi_{K}^{c},\Zbb_{\ell}). \]
Notice that the
kernel of the $\ell$-adic completion map  $K^\times\to\hat{K^\times}$ is $k^\times$, hence it induces an embedding
$K^{\times}\!/k^{\times} \hookrightarrow \hat{K^{\times}}$, while $K^{\times}\!/k^{\times}$ can be identified with the projectivization of the $k$-vector space $(K,+)$.
The strategy of BP now proceeds as follows:
\begin{enumerate}
  \item First, using the projections $\pi_t:\Pi^c_K\to\Pi^c_{\Pbb^1\smin\{0,1,\infty\}}$,
  one identifies the prime divisors of $K|k$ (among the quasi-prime divisors described via the \textit{local theory}).
  \item Second, one recovers $k(u)^{\times}\!/k^{\times}\hookrightarrow\hat{K^{\times}}$
  for all $u\in K\smin k$, thus $K^{\times}/k^{\times}\hookrightarrow\hat{K^{\times}}$ as a subgroup.
  One also recovers the \emph{projective lines}
  $\lfrak_{f,g}:=(kf+kg)^\times\!/k^\times\subset K^\times\!/k^\times$, where $f,g\in K$ are $k$-linearly independent, as being $\lfrak_{f,g}= f \cdot \lfrak_{1,u}$ with $u=g/f$.
  \item Finally, apply the fundamental theorem of projective geometry~\cite{artin} to obtain $(K,+)$ as a $k$-vector space, and use the compatibility with the multiplicative structure of $K^{\times}\!/k^{\times}$ to obtain the field structure of $K$.
\end{enumerate}
Moreover, one shows that the recipes in steps (1), (2) and (3) are invariant under the action of $\PAut^c(\Pi^a_{\Vcal_0^{\bir}})$.
Thus one obtains a morphism
\[ \PAut^c(\Pi^a_{\Vcal_0^{\bir}}) \to \Aut_{\Sigma_0}(K) = \Gal(\Qbar|\Qbb) \]
which is then shown to be the inverse of the canonical map $\Gal(\Qbar|\Qbb) \to \PAut^{c}(\Pi^{a}_{\Vcal_{0}^{\bir}})$.

In practice, both steps (1) and (2) above rely on the so-called \emph{local theory} whereby one detects inertia and decomposition groups in $\Pi_{K}^{a}$, associated to quasi-prime divisors of $K|k$.
This local theory builds on the \emph{theory of rigid elements}~\cite{ArasonElmanJacob} which exploits the {\it field structure\/} of $K$; see~\cite{Topaz2017}. 
Furthermore, in order to recover $\lfrak_{f,g}$ in step~(2), one uses
$\lfrak_{1,u}$ with $u=g/f$.
Hence in step (2) one also relies on the full field structure of $K$, see~\cite{Pop2012} for more on the general global theory.
Thus, both the local and global parts of the BP strategy used to obtain the equality $\Gal(\Qbar|\Qbb) = \Aut^c(\Pi^a_{\Vcal_{0}^\bir})$ \emph{depend in an essential way} on being in a \emph{birational context}.

As outlined in \S\ref{subsection:sketch}, we still have a \emph{local} and a \emph{global} portion to the proof of our Main Theorem, and we again eventually rely on the fundamental theorem of projective geometry~\cite{artin}.
However, if one mimics the BP strategy outlined in the steps above, then Kummer Theory yields the subgroup of $\hat{K^\times}$ generated by Kummer classes of functions of the form $a + b x + c y \in K$, $a,b,c \in k$. 
While this subgroup is contained in $K^\times\!/k^\times$, it is \emph{not a projective-linear subspace}, and thus the BP strategy breaks down.
In other words, there is no (obvious or non-obvious) candidate for a $k$-projective space arising from Kummer Theory that plays the role of $K^\times\!/k^\times$.
Therefore we had to develop some genuinely new techniques for both the local and global portions of our argument which do not rely on the arithmetic structure of function fields of any of the varieties involved, and which do not apply the fundamental theorem of projective geometry on an object constructed using Kummer theory.

\section{Preparation and results}

We now introduce the notation and terminology necessary to state our main result.

\subsection{Hyperplane arrangements}
Let $k$ be an algebraically closed field which will be fixed throughout.
By a \emph{variety}, we mean a $k$-variety, i.e.~an integral scheme which is separated and of finite type over $k$.
We will always omit $k$ from the notation whenever possible.
For example, we write $\Abb^{n}$ for $\Abb^{n}_{k}$, $\Pbb^{n}$ for $\Pbb^{n}_{k}$, etc.

A morphism of varieties is a morphism of $k$-schemes and we denote by $\Var$ the category of varieties.
If $k_{0}$ is a subfield of $k$, we say that a variety $X$ resp.~a morphism of varieties $X \to Y$ is \emph{defined over $k_{0}$} if it is the base-change of some $k_{0}$-variety resp.~some morphism of $k_{0}$-varieties.
If $X$ is a variety, then by a \emph{closed subvariety} of $X$ we mean a reduced closed subscheme of $X$; in particular, our convention is that subvarieties may have many irreducible components.
If $Z \subset X$ is a closed subvariety and $X$ is the base-change of a $k_{0}$-scheme $X_{0}$, then we say that $Z$ is \emph{defined over $k_{0}$} provided that $Z$ is the base-change of some closed subscheme $Z_{0} \subset X_{0}$.

Let $\Afrak$ be a variety and $x_{1},\ldots,x_{n} \in \Ocal(\Afrak)$ be given.
Put $\xbf := (x_{1},\ldots,x_{n})$ and consider the induced map $\Afrak \to \Abb^{n}$ defined by $t_{i} \mapsto x_{i}$.
We will say that the pair $(\Afrak,\xbf)$ is an \emph{affine space} provided that this map is an isomorphism, and in this case call $\xbf$ a \emph{system of coordinates} on $\Afrak$.
We will usually omit $\xbf$ from the notation when referring to affine spaces, and simply write $\Afrak$ instead of $(\Afrak,\xbf)$.

A \emph{partial system of coordinates} $\zbf = (z_{1},\ldots,z_{m})$ on a variety $\Afrak$ is a tuple which can be extended to a system of coordinates.
If $\xbf = (x_{1},\ldots,x_{m})$ is a (partial) system of coordinates, we write $\Abb^{m}_{\xbf}$ for $\Spec[x_{1},\ldots,x_{m}]$ with the $x_{i}$ considered as indeterminant variables, and write $\pi_{\xbf} : \Afrak \to \Abb^{m}_{\xbf}$ for the associated projection which sends the element $x_{i} \in k[x_{1},\ldots,x_{m}]$ to $x_{i} \in \Ocal(\Afrak)$.
If $\xbf = (\varpi)$ consists of a single element, we abbreviate the notation as $\Abb^{1}_{\varpi} := \Abb^{1}_{\xbf}$ and $\pi_{\varpi} : \Afrak \to \Abb^{1}_{\varpi}$ for the associated projection.

Let $\Afrak$ be an affine space with system of coordinates $\xbf = (x_{1},\ldots,x_{n})$.
By a \emph{hyperplane} we mean a closed subvariety of $\Afrak$ which is the zero-locus of a function of the form
\[ a_{0} + a_{1} \cdot x_{1} + \cdots + a_{n} \cdot x_{n}, \ \ a_{i} \in k, \ (a_{1},\ldots,a_{n}) \ne 0. \]
Note that the notion of a hyperplane in $\Afrak$ depends implicitly on a choice of system of coordinates $\xbf$, and we will ensure that $\xbf$ is clear from context whenever we speak about hyperplanes.

Such hyperplanes will usually be denoted using the letter $\Hfrak$ possibly decorated in some way.
If $\varpi := a_{0} + a_{1} \cdot x_{1} + \cdots + a_{n} \cdot x_{n}$ is a function as above, we will write $\Hfrak_{\varpi}$ for the associated hyperplane obtained as the zero-locus of $\varpi$.
The collection of all hyperplanes in an affine space $\Afrak$ will be denoted by $\Planes_{\Afrak}$, or just $\Planes$ if $\Afrak$ is understood from context.

A \emph{hyperplane arrangement} in an affine space $\Afrak$ is a finite union of the form
\[ \Hfrak_{1} \cup \cdots \cup \Hfrak_{n}, \ \Hfrak_{i} \in \Planes_{\Afrak}, \]
considered as a reduced closed subvariety of $\Afrak$.
Hyperplane arrangements will usually be denoted using the letter $\Hbf$ possibly decorated in some way.
Given a hyperplane arrangement $\Hbf$ in $\Afrak$, we write $\Ucal_{\Hbf} := \Afrak \smin \Hbf$ for its complement.

If $\xbf = (x_{1},\ldots,x_{n})$ is a system of coordinates for an affine space $\Afrak$, $\varpi \in k[x_{1},\ldots,x_{n}]$ is a linear polynomial whose zero-locus is a hyperplane $\Hfrak_{\varpi}$, and $\Hbf$ is a hyperplane arrangement in $\Afrak$ which contains $\Hfrak_{\varpi}$, then $\pi_{\varpi} : \Afrak \to \Abb^{1}_{\varpi}$ restricts to a morphism $\pi_{\varpi} : \Ucal_{\Hbf} \to \Gbb_{m}$.
If $\Hbf$ also contains $\Hfrak_{1-\varpi}$, then we may restrict further to obtain $\pi_{\varpi} : \Ucal_{\Hbf} \to \Pbb^{1} \smin \{0,1,\infty\}$.

\subsection{The category $\Hscr$}
Let $k_{0}$ be a perfect field with algebraic closure $k$, and $S \subset k_{0}$ any subset with $0 \in S$.
Let $\Afrak$ be an affine space with system of coordinates $\xbf = (x_{1},\ldots,x_{n})$.
We define a subcategory $\Hscr_{k_{0},\Afrak}$ of $\Var$ as follows.
The objects of $\Hscr_{k_{0},\Afrak}$ are $\Ucal_{\Hbf}$ for $\Hbf$  hyperplane arrangements in $\Afrak$ which are defined over $k_{0}$ (as a closed subvariety of $\Afrak$).
The morphisms in $\Hscr_{k_{0},\Afrak}$ are the inclusions $\Ucal_{\Hbf_{1}} \hookrightarrow \Ucal_{\Hbf_{2}}$ whenever $\Hbf_{2} \subset \Hbf_{1}$.
We write $\Hscr_{\xbf,S}$ for the smallest subcategory of $\Var$ containing $\Hscr_{k_0,\Afrak}$, the object $\Pbb^{1} \smin \{0,1,\infty\}$, and the projections
\[ \pi_{\varpi} : \Ucal_{\Hbf} \to \Pbb^{1}_{k} \smin \{0,1,\infty\} \]
for $\varpi$ any function of the form $x_{i}-c$ for $1 \le i \le n$ and $c \in S$, or of the form $x_{i} - x_{j}$ for $1 \le i < j \le n$, whenever $\Hbf$ contains the hyperplanes $\Hfrak_{\varpi}$ and $\Hfrak_{\varpi-1}$.
Note that the objects and morphisms of $\Hscr_{\xbf,S}$ are all defined over $k_{0}$.

\subsection{Fundamental groups}\label{subsec:intro-fund-groups}
Let $\ell$ be a prime different from the characteristic of $k$, and $\Lambda$ any nontrivial quotient of $\Zbb_{\ell}$.
Both $\ell$ and $\Lambda$ will be fixed from now on, and we shall write $\HH^{i}(-)$ instead of $\HH^{i}(-,\Lambda)$.
For a profinite group $\Pi$, write $\Pi^{(2)}$ for the left kernel of the canonical pairing
\[ \Pi \times \HH^{1}(\Pi) \to \Lambda, \]
and put $\Pi^{a} = \Pi/\Pi^{(2)}$.
Inflation provides a canonical isomorphism
\[ \HH^{1}(\Pi^{a}) \cong \HH^{1}(\Pi) \]
and so the usual five-term exact sequence restricts to an exact sequence of the form
\[ 0 \to \HH^{1}(\Pi^{(2)})^{\Pi} \xrightarrow{d_{2}} \HH^{2}(\Pi^{a}) \to \HH^{2}(\Pi). \]
Let $\HH^{2}(\Pi^{a})_{\dec}$ denote the submodule of $\HH^{2}(\Pi^{a})$ generated by cup-products of elements of $\HH^{1}(\Pi^{a})$, and let $\HH^{1}(\Pi^{2})^{\Pi}_{0}$ denote its preimage in $\HH^{1}(\Pi^{(2)})^{\Pi}$.
We shall write $\Pi^{(3)}$ for the left kernel of the canonical pairing
\[ \Pi^{(2)} \times \HH^{1}(\Pi^{(2)})^{\Pi}_{0} \to \Lambda \]
and put $\Pi^{c} := \Pi/\Pi^{(3)}$.

For a $k$-variety $X$, we write $\Pi_{X}$ for the \'etale fundamental group of $X$ with respect to some base point, and $\Pi_{X}^{c}$ resp.~$\Pi_{X}^{a}$ for the quotients of $\Pi_{X}$ as defined above.
For $\star \in \{a,c\}$, the object $\Pi_{X}^{\star}$ is functorial in $X$, taking values in the category $\Profout$.
Whenever $\Vcal$ is a subcategory of $\Var$, we write $\Pi_{\Vcal}^{\star}$ for the restriction of the functor $X \mapsto \Pi_{X}^{\star}$ to $\Vcal$.
Note that we have a natural surjective morphism $\Pi_{\Vcal}^{c} \twoheadrightarrow \Pi_{\Vcal}^{a}$, and any automorphism of $\Pi_{\Vcal}^{c}$ induces a compatible automorphism of $\Pi_{\Vcal}^{a}$.
We write
\[ \Aut^{c}(\Pi_{\Vcal}^{a}) := \image\left(\Aut(\Pi_{\Vcal}^{c}) \to \Aut(\Pi_{\Vcal}^{a})\right) \]
for the group of automorphisms of $\Pi_{\Vcal}^{a}$ which lift to an automorphism of $\Pi^{c}_{\Vcal}$.

Note that $\Lambda^{\times}$ acts on the functor $\Pi_{\Vcal}^{a}$ canonically, namely
$\epsilon \in \Lambda^{\times}$ acts on $\Pi_{X}^{a}$ by multiplication since it is, in particular, a $\Lambda$-module.
We put $\PAut(\Pi^{a}_{\Vcal}) := \Aut(\Pi^{a}_{\Vcal})/\Lambda^{\times}$, and write $\PAut^{c}(\Pi_{\Vcal}^{a})$ for the image of $\Aut^{c}(\Pi_{\Vcal}^{a})$ in $\PAut(\Pi^{a}_{\Vcal})$.

Suppose now that $k_{0}$ is a perfect subfield of $k$ whose algebraic closure is $k$.
If the objects and morphisms in $\Vcal$ are all defined over $k_{0}$, then functoriality provides us with a canonical morphism
\[ \rho : \Gal(k|k_{0}) \to \Aut^{c}(\Pi^{a}_{\Vcal}) \twoheadrightarrow \PAut^{c}(\Pi^{a}_{\Vcal}). \]

\subsection{Main result}

The Main Theorem stated in subsection~1.2
%page~\pageref{maintheorem:intro} 
is the special case of the following general result, where $k_{0} = \Qbb$, $n = 2$, $S = \{0\}$ and $\Lambda = \Zbb_{\ell}$.
\begin{maintheorem}\label{maintheorem:general}
  Let $k_{0}$ be a perfect field of characteristic $\neq \ell$ with algebraic closure $k$, and $S$ a set of generators of $k_{0}$ as a field extension of its prime subfield, with $0 \in S$.
  Let $\Afrak$ be an affine space of dimension at least two with system of coordinates $\xbf = (x_{1},\ldots,x_{n})$.
  Then the canonical map
  \[ \rho : \Gal(k|k_{0}) \to \PAut^{c}(\Pi_{\Hscr_{\xbf,S}}^{a}) \]
  is an isomorphism.
\end{maintheorem}

\section{Cohomology}

Throughout this work, we write $\HH^{i}(-) := \HH^{i}(-,\Lambda(i))$ for the $i$-th (geometric) \'etale cohomology group ($\ell$-adic cohomology in the case $\Lambda = \Zbb_{\ell}$) with values in the Tate twist $\Lambda(i)$.
We will specify the coefficients in cohomology if they differ from the convention above.

\subsection{Kummer classes of hyperplanes}\label{subsec:kummer-hyperplanes}

Let $\Afrak$ be an affine space with system of coordinates $\xbf = (x_{1},\ldots,x_{n})$.
Let $\Hfrak$ a hyperplane in $\Afrak$ and $\varpi \in k[x_{1},\ldots,x_{n}]$ be a linear polynomial whose zero-locus is $\Hfrak$.
Recall that $\varpi$ can be considered as a morphism $\pi_{\varpi} : \Afrak \to \Abb^{1}$ which restricts to a morphism $\pi_{\varpi} : \Afrak \smin \Hfrak \to \Gbb_{m}$.
If $U$ is any nonempty open subset of $\Abb^{1}$ and $V$ any nonempty open subset of $\pi_{\varpi}^{-1}(U)$, then we denote by $\pi_{\varpi} : V \to U$ the morphism induced by restricting $\pi_{\varpi}$ and
\[ \iota_{\Hfrak} : \HH^{1}(U) \to \HH^{1}(V) \]
the corresponding morphism on cohomology.
These morphisms are of course compatible with restriction along open sets.

In the case where $U = \Gbb_{m} = \Abb^{1} \smin \{0\}$, hence $\HH^{1}(U) = \HH^{1}(\Gbb_{m}) = \Lambda$, we write $[\Hfrak] := \iota_{\Hfrak}(1)$ and call $[\Hfrak]$ the \emph{Kummer class} associated to $\Hfrak$.
As the notation suggests, $[\Hfrak] \in \HH^{1}(V)$ only depends on $\Hfrak$, and not on the choice of $\xbf$ or $\varpi$ whose zero-locus is $\Hfrak$.
The following well-known fact follows from the rationality of $\Afrak$.

\begin{fact}\label{fact:H1-residue-description}
  Let $\Hbf$ be a hyperplane arrangement in $\Afrak$ with distinct irreducible components $\Hfrak_{1},\ldots,\Hfrak_{n}$.
  Then the set $\{[\Hfrak_{1}],\ldots,[\Hfrak_{n}]\}$ forms a basis for $\HH^{1}(\Ucal_{\Hbf})$.
  Furthermore, the residue maps $\partial_{\Hfrak_{i}} : \HH^{1}(\Ucal_{\Hbf}) \to \HH^{0}(\Ucal_{\Hbf} \cap \Hfrak_{i}) = \Lambda$ associated to $\Hfrak_{i}$ satisfy $\partial_{\Hfrak_{i}}[\Hfrak_{j}] = \delta_{i,j}$ where $\delta_{i,j}$ denotes the Kronecker $\delta$-function taking values in $\Lambda$.
\end{fact}
\begin{proof}
  Identify $\Afrak$ with $\Abb^{n}$, embed $\Abb^{n}$ into $\Pbb^{n}$ in the usual way, and write $\Hfrak_{\infty}$ for the hyperplane of $\Pbb^{n}$ at infinity.
  Write $\tilde\Hbf$ for the closure of $\Hbf$ in $\Pbb^{n}$ and $\tilde\Ucal_{\Hbf}$ for the complement of $\tilde\Hbf$.
  By cohomological purity, the inclusion $\Ucal_{\Hbf}\hookrightarrow \Afrak \hookrightarrow \Pbb^{n}$ induces an exact sequence of the form
  \[ 0 \to \HH^{1}(\Pbb^{n}) \to \HH^{1}(\Ucal_{\Hfrak}) \to \HH^{0}(\Hfrak_{\infty} \cap \tilde\Ucal_{\Hbf}) \oplus \bigoplus_{i} \HH^{0}(\Hfrak \cap \Ucal_{\Hbf}) \to \HH^{2}(\Pbb^{n},\Lambda(1)). \]
  Now $\HH^{2}(\Pbb^{n},\Lambda(1)) = \Lambda$, all the $\HH^{0}$ terms appearing in this sequence are also $\Lambda$, and the map on the right is simply the sum.
  Since $\HH^{1}(\Pbb^{n}) = 0$, the assertion follows.
\end{proof}

\subsection{Relations}

We work with hyperplanes inside a fixed affine space $\Afrak$ in this subsection.
Note that the intersection of two hyperplanes $\Hfrak_{1},\Hfrak_{2}$ in $\Afrak$ is either empty or has codimension two.
A pair of hyperplanes $(\Hfrak_{1},\Hfrak_{2})$ will be called a \emph{parallel pair} provided $\Hfrak_{1} \ne \Hfrak_{2}$ and $\Hfrak_{1} \cap \Hfrak_{2} = \varnothing$.
A triple of hyperplanes $(\Hfrak_{1},\Hfrak_{2},\Hfrak_{3})$ will be called a \emph{dependent triple} provided $\Hfrak_{1},\Hfrak_{2},\Hfrak_{3}$ are distinct and $\Hfrak_{1} \cap \Hfrak_{2} \cap \Hfrak_{3}$ has codimension two in $\Afrak$.

\begin{lemma}\label{lemma:cup-product-relations}
  Let $\Hbf$ be a hyperplane arrangement in $\Afrak$ whose distinct irreducible components are $\Hfrak_{1},\ldots,\Hfrak_{n}$.
  Then the following relations hold in $\HH^{2}(\Ucal_{\Hbf})$:
  \begin{enumerate}
    \item If $(\Hfrak_{i},\Hfrak_{j})$ form a parallel pair, then $[\Hfrak_{i}] \cup [\Hfrak_{j}] = 0$.
    \item If $(\Hfrak_{i},\Hfrak_{j},\Hfrak_{k})$ are a dependent triple, then one has $([\Hfrak_{i}] - [\Hfrak_{k}]) \cup ([\Hfrak_{j}] - [\Hfrak_{k}]) = 0$.
  \end{enumerate}
\end{lemma}
\begin{proof}
  To obtain (2), simply note that whenever $(\Hfrak_{i},\Hfrak_{j},\Hfrak_{k})$ is a dependent triple and $\Hbf'$ is the hyperplane arrangement obtained from $\Hbf$ as the closure in $\Afrak$ of
  \[ \Hbf \smin \Hfrak_{i} \cup \Hfrak_{j} \cup \Hfrak_{k}, \]
  then there exists a linear morphism $f : \Ucal_{\Hbf'} \to \Pbb^{1}$ such that $\Hfrak_{i}$ is the fiber above $0$, $\Hfrak_{j}$ is the fiber above $1$ and $\Hfrak_{k}$ is the fiber above $\infty$.
  We thus obtain a restricted linear morphism $f : \Ucal_{\Hbf} \to \Pbb^{1} \smin \{0,1,\infty\} = \Abb^{1} \smin \{0,1\}$.
  The pullbacks with respect to $f$ of the Kummer classes $[0],[1] \in \HH^{1}(\Abb^{1} \smin \{0,1\})$ satisfy:
  \[ f^{*}[0] = [\Hfrak_{i}] - [\Hfrak_{k}], \ f^{*}[1] = [\Hfrak_{j}] - [\Hfrak_{k}]. \]
  Assertion (2) follows since $\HH^{2}(\Pbb^{1} \smin \{0,1,\infty\})$ vanishes.

  Assertion (1) is obtained similarly by identifying $\Afrak$ with $\Abb^{n}$ via some choice of coordinates so that $\Ucal_{\Hbf}$ can be identified as the complement of $\Pbb^{n}$ of $\Hbf \cup \Hfrak_{\infty}$ where $\Hfrak_{\infty}$ is the (projective) hyperplane at infinity.
  The argument for case (2) above goes through, \textit{mutatis mutandis}, with $\Hfrak_{\infty}$ in place of $\Hfrak_{k}$.
\end{proof}

\begin{remark}
  The relations appearing in Lemma~\ref{lemma:cup-product-relations} are well-known and are used in defining \emph{Orlik-Solomon} algebras associated to hyperplane arrangements.
\end{remark}

\subsection{The universal case}\label{subsec:univ-case-cohom}

We now pass to colimits over certain Zariski open subsets.
If $\Afrak$ is any affine space, we define
\[ \HHbf^{*}_{\Afrak} := \varinjlim_{\Ucal_{\Hbf}} \HH^{*}(\Ucal_{\Hbf})\]
where $\Ucal_{\Hbf}$ varies over $\Hscr_{k_{0},\Afrak}$.
If $\xbf := (x_{1},\ldots,x_{n})$ is a system of coordinate on $\Afrak$, we will also write $\HHbf^{*}_{\xbf} := \Hbf^{*}_{\Afrak}$.
In the case where $\xbf = (\varpi)$ is a singleton, we write $\HHbf^{*}_{\varpi} := \HHbf^{*}_{\xbf}$.

We have an obvious notion of Kummer classes $[\Hfrak] \in \HHbf^{1}_{\Afrak}$ associated to hyperplanes $\Hfrak$ of $\Afrak$, and the discussion of \S\ref{subsec:kummer-hyperplanes} shows that these Kummer classes form a basis for $\HHbf^{1}_{\Afrak}$.
For $\Hbf$ as above, the canonical morphism $\HH^{1}(\Ucal_{\Hbf}) \to \HHbf^{1}_{\Afrak}$ is \emph{injective} and its image is generated by $[\Hfrak]$ for $\Hfrak$ varying over the irreducible components of $\Hbf$.
The following relations involving Kummer classes of hyperplanes in $\Afrak$ hold true in $\HHbf^{2}_{\Afrak}$ by Lemma~\ref{lemma:cup-product-relations}:
\begin{enumerate}
  \item $[\Hfrak_{1}] \cup [\Hfrak_{2}] = 0$ for parallel pairs $(\Hfrak_{1},\Hfrak_{2})$.
  \item $([\Hfrak_{1}] - [\Hfrak_{3}]) \cup ([\Hfrak_{2}] - [\Hfrak_{3}]) = 0$ for dependent triples $(\Hfrak_{1},\Hfrak_{2},\Hfrak_{3})$.
\end{enumerate}

\subsection{Linear projections}

Let $\Afrak$ be an affine space and $\zbf = (z_{1},\ldots,z_{m})$ a partial system of coordinates on $\Afrak$.
Consider the induced morphism $\pi_{\zbf} : \Afrak \to \Abb^{m}_{\zbf}$.
Restricting to the appropriate open subsets, and passing to cohomology and the colimit, we obtain a canonical morphism
\[ \iotabf_{\zbf} : \HHbf^{*}_{\zbf} \to \HHbf^{*}_{\Afrak}. \]
If $\Hfrak$ is a hyperplane in $\Abb^{m}_{\zbf}$, then $\pi_{\zbf}^{-1}\Hfrak$ is a hyperplane in $\Afrak$ and their Kummer classes are compatible in the sense that one has $\iotabf_{\zbf}[\Hfrak] = [\pi_{\zbf}^{-1}\Hfrak]$.
It follows easily from this that $\iotabf_{\zbf}$ is injective and its image is generated by Kummer classes of the form $[\pi_{\zbf}^{-1}\Hfrak]$ for $\Hfrak$ varying over the hyperplanes in $\Abb^{m}_{\zbf}$.

\subsection{Residue maps}\label{subsec:univ-residue}

Let $\Afrak$ be an affine space and $\Hfrak \subset \Afrak$ a hyperplane.
We have a canonical residue map
\[ \partial_{\Hfrak} : \HHbf^{i+1}_{\Afrak} \to \HHbf^{i}_{\Hfrak} \]
obtained from the usual residue maps associated to $\Hfrak$ by passing to the colimit.
These residue maps can be calculated using the following formulas:
\begin{enumerate}
  \item One has $\partial_{\Hfrak}[\Hfrak] = 1$.
  \item One has $\partial_{\Hfrak}[\Hfrak'] = 0$ for $\Hfrak'$ a hyperplane distinct from $\Hfrak$.
  \item If $\Hfrak'$ is a hyperplane and $\Hfrak \cap \Hfrak'$ is a hyperplane in $\Hfrak$, then $\partial_{\Hfrak}([\Hfrak] \cup [\Hfrak']) = [\Hfrak \cap \Hfrak']$.
\end{enumerate}
We write
\[ \Ubf_{\Hfrak} := \ker(\partial_{\Hfrak} : \HHbf^{1}_{\Afrak} \to \HHbf^{0}_{\Hfrak}). \]
We have a canonical \emph{specialization map} $\sfrak_{\Hfrak} : \Ubf_{\Hfrak} \to \HHbf^{1}_{\Hfrak}$ defined on the level of cohomology by restriction.
Explicitly, this specialization map satisfies $\sfrak_{\Hfrak}[\Hfrak'] = \partial_{\Hfrak}([\Hfrak] \cup [\Hfrak'])$.
We define
\[ \Ubf^{1}_{\Hfrak} := \ker(\sfrak_{\Hfrak} : \Ubf_{\Hfrak} \to \HHbf^{1}_{\Hfrak}). \]
Clearly, $\sfrak_{\Hfrak}$ is surjective and thus it induces an isomorphism $\Ubf_{\Hfrak}/\Ubf_{\Hfrak}^{1} \cong \HHbf^{1}_{\Hfrak}$.
The submodules $\Ubf^{1}_{\Hfrak} \subset \Ubf_{\Hfrak}$ can be described explicitly using the description of $\partial_{\Hfrak}$ above, as follows.
\begin{fact}\label{fact:Ubf-description}
  The following hold:
  \begin{enumerate}
    \item $\Ubf_{\Hfrak}$ is generated by Kummer classes of the form $[\Hfrak_{1}]$ for hyperplanes $\Hfrak_{1} \neq \Hfrak$.
    \item $\Ubf^{1}_{\Hfrak}$ is generated by the following:
          \begin{enumerate}
            \item Kummer classes $[\Hfrak_{1}]$ where $(\Hfrak,\Hfrak_{1})$ is a parallel pair.
            \item Differences $[\Hfrak_1] - [\Hfrak_2]$ where $(\Hfrak,\Hfrak_1,\Hfrak_2)$ is a dependent triple.
          \end{enumerate}
  \end{enumerate}
\end{fact}
\begin{proof}
  These properties follow from the explicit description of $\HH^{1}$ in terms of residue maps appearing in Fact~\ref{fact:H1-residue-description}.
\end{proof}

\section{Fundamental groups}\label{section:fund-groups}

For a variety $X$ let $\Pi_{X}$, $\Pi_{X}^{(2)}$, $\Pi_{X}^{(3)}$, $\Pi_{X}^{a}$ and $\Pi_{X}^{c}$ be as defined in \S\ref{subsec:intro-fund-groups}.
We also put $\Pi_{X}^{\delta} := \Pi_{X}^{(2)}/\Pi_{X}^{(3)}$.
These groups fit in a central extension of the form
\[ 1 \to  \Pi_{X}^{\delta} \to \Pi_{X}^{c} \to \Pi_{X}^{a} \to 1, \]
with both $\Pi_{X}^{a}$ and $\Pi_{X}^{\delta}$ being $\Lambda$-modules.
The commutator in $\Pi_{X}^{c}$ induces a $\Lambda$-bilinear map
\[ [-,-] : \wedge^{2}\Pi_{X}^{a} \to \Pi_{X}^{\delta} \]
defined as $[\sigma,\tau] := \tilde\sigma \cdot \tilde\tau \cdot \tilde\sigma^{-1} \cdot \tilde\tau^{-1}$, where $\tilde\sigma,\tilde\tau \in \Pi_{X}^{c}$ denote lifts of $\sigma,\tau \in \Pi_{X}^{a}$.

\subsection{Duality}\label{subsec:duality}

Let $X$ be a smooth quasi-projective variety which is isomorphic to the complement of finitely many hyperplanes in an affine space.
There exists a well-known duality between the 2-truncation of $\HH^{*}(X)$ and $\Pi_{X}^{c}$, which we now summarize.
Write $\Rcal(X)$ for the kernel of the cup-product map $\wedge^{2} \HH^{1}(X) \to \HH^{2}(X)$.
We have a perfect pairing of the form
\begin{equation}\label{eqn:H1-pairing}
  \Pi^{a}_{X} \times \HH^{1}(X) \to \Lambda.
\end{equation}
Indeed, in the case where $\Lambda$ is finite this is a consequence of Pontryagin duality, while the case $\Lambda = \Zbb_{\ell}$ follows by passing to the limit over the finite quotients of $\Zbb_{\ell}$ and using the explicit description of $\HH^{1}(X)$ from \S\ref{subsec:kummer-hyperplanes}.
Note that both $\HH^{1}(X)$ and $\Pi^{a}_{X}$ are finitely-generated free $\Lambda$-modules of the same rank.
The dual with respect to~\eqref{eqn:H1-pairing} of the inclusion $\Rcal(X) \hookrightarrow \wedge^{2} \HH^{1}(X)$ factors through the commutator $[-,-] : \wedge^{2}\Pi_{X}^{a} \to \Pi_{X}^{\delta}$.
Thus we have a natural pairing
\begin{equation}\label{eqn:H2-pairing}
  \Pi^{\delta}_{X} \times \Rcal(X) \to \Lambda
\end{equation}
Which is compatible with~\eqref{eqn:H1-pairing} as described above.

\subsection{The universal case}\label{subsec:univ-duality}

We now pass to limits analogously to the discussion from \S\ref{subsec:univ-case-cohom}.
Let $\Afrak$ be an affine space.
We define
\[ \Pibf^{\star}_{\Afrak} := \varprojlim_{\Ucal_{\Hbf}} \Pi^{\star}_{\Ucal_{\Hbf}}, \ \star \in \{a,c\}, \ \Pibf^{\delta}_{\Afrak} := \ker(\Pibf^{c}_{\Afrak} \to \Pibf^{a}_{\Afrak}), \]
where $\Ucal_{\Hbf}$ varies over $\Hscr_{k_{0},\Afrak}$.
The pairing~\eqref{eqn:H1-pairing} extends to a perfect pairing
\begin{equation}\label{eqn:univ-H1-pairing}
  \Pibf^{a}_{\Afrak} \times \HHbf^{1}_{\Afrak} \to \Lambda.
\end{equation}
If $\xbf = (x_{1},\ldots,x_{n})$ is a system of coordinates on $\Afrak$, we may write $\Pibf^{\star}_{\xbf} := \Pibf^{\star}_{\Afrak}$, $\star \in \{a,c,\delta\}$, and if $\xbf = (\varpi)$ is a singleton, we write $\Pibf^{\star}_{\varpi} := \Pibf^{\star}_{\xbf}$, $\star \in \{a,c,\delta\}$.
Given an element $\alpha \in \HHbf^{1}_{\Afrak}$ and $\sigma \in \Pibf^{a}_{\Afrak}$, we write $\sigma\alpha \in \Lambda$ for the image of $(\sigma,\alpha)$ under the pairing~\eqref{eqn:univ-H1-pairing}.

The commutator $[-,-] : \hat\wedge^{2} \Pibf^{a}_{\Afrak} \to \Pibf^{\delta}_{\Afrak}$ is not injective in general.
Describing its kernel boils down to dualizing the cup-product $\wedge^{2}\HHbf^{1}_{\Afrak} \to \HHbf^{2}_{\Afrak}$ and using the relations among cup-products of Kummer classes discussed in \S\ref{subsec:univ-case-cohom}.
We will only need the following properties.

\begin{lemma}\label{lemma:alt-pairs}
  Let $\sigma,\tau \in \Pibf^{a}_{\Afrak}$ be given and assume that $[\sigma,\tau] = 0$ in $\Pibf^{\delta}_{\Afrak}$.
  \begin{enumerate}
    \item For all parallel pairs $(\Hfrak_{1},\Hfrak_{2})$, one has
          \[ \sigma[\Hfrak_{1}] \cdot \tau[\Hfrak_{2}] = \sigma[\Hfrak_{2}] \cdot \tau[\Hfrak_{1}]. \]
    \item For all dependent triples $(\Hfrak_{1},\Hfrak_{2},\Hfrak_{3})$, one has
          \[ (\sigma[\Hfrak_{1}] - \sigma[\Hfrak_{3}]) \cdot (\tau[\Hfrak_{2}] - \tau[\Hfrak_{3}]) = (\sigma[\Hfrak_{2}] - \sigma[\Hfrak_{3}]) \cdot (\tau[\Hfrak_{1}] - \tau[\Hfrak_{3}]). \]
  \end{enumerate}
\end{lemma}
\begin{proof}
  This follows from the compatibility of the pairings~\eqref{eqn:H1-pairing} and~\eqref{eqn:H2-pairing}, and Lemma~\ref{lemma:cup-product-relations}.
  See the discussion of \S\ref{subsec:univ-case-cohom}.
\end{proof}

\subsection{Linear projections}
Let $\zbf = (z_{1},\ldots,z_{m})$ be a partial system of coordinates in an affine space $\Afrak$.
Consider the morphism $\pi_{\zbf} : \Afrak \to \Abb^{m}_{\zbf}$.
By restricting to the appropriate open subsets, applying the functor $X \mapsto \Pi_{X}^{\star}$, $\star \in \{a,c\}$, and passing to the limit, we obtain canonical morphisms
\[ \pibf_{\zbf} : \Pibf^{\star}_{\Afrak} \to \Pibf^{\star}_{\zbf}. \]
The map $\pibf_{\zbf} : \Pibf^{a}_{\Afrak} \to \Pibf^{a}_{\zbf}$ is easily seen to be dual to $\iotabf_{\zbf} : \HHbf^{1}_{\zbf} \to \HHbf^{1}_{\Afrak}$ with respect to~\eqref{eqn:univ-H1-pairing}.

\section{The local theory}

The primary goal of this section is to provide a characterization of hyperplanes in an affine space $\Afrak$ using decomposition theory within $\Pibf^{a}_{\Afrak}$.

\subsection{Inertia and decomposition groups}\label{subsec:inertia-decomp}

Let $\Hfrak$ be a hyperplane in an affine space $\Afrak$ and recall that we have introduced two submodules $\Ubf_{\Hfrak}^{1} \subset \Ubf_{\Hfrak} \subset \HHbf^{1}_{\Afrak}$ in \S\ref{subsec:univ-residue}.
We define $\Ibf_{\Hfrak}$ resp.~$\Dbf_{\Hfrak}$ as the orthogonal of $\Ubf_{\Hfrak}$ resp.~$\Ubf_{\Hfrak}^{1}$ with respect to the pairing~\eqref{eqn:univ-H1-pairing}.
Note that $\Ibf_{\Hfrak} \subset \Dbf_{\Hfrak}$, and the quotient $\Dbf_{\Hfrak}/\Ibf_{\Hfrak}$ is canonically isomorphic to $\Pibf^{a}_{\Hfrak}$ because of the specialization isomorphism $\Ubf_{\Hfrak}/\Ubf_{\Hfrak}^{1} \cong \HHbf^{1}_{\Hfrak}$ and the duality described in~\eqref{eqn:univ-H1-pairing}.
Note that $\Ibf_{\Hfrak} \cong \Lambda$ since $\HHbf^{1}_{\Afrak}/\Ubf_{\Hfrak} \cong \Lambda$.

Explicitly, $\Ibf_{\Hfrak}$ is the space of all $\sigma \in \Pibf^{a}_{\Afrak}$ such that $\sigma[\Hfrak_1] = 0$ for all hyperplanes $\Hfrak_1$ distinct from $\Hfrak$.
Similarly $\Dbf_{\Hfrak}$ is the space of all $\sigma \in \Pibf^{a}_{\Afrak}$ such that for all parallel pairs of the form $(\Hfrak,\Hfrak_1)$, one has $\sigma[\Hfrak_1] = 0$, and for all dependent triples of the form $(\Hfrak,\Hfrak_1,\Hfrak_2)$, one has $\sigma[\Hfrak_1] = \sigma[\Hfrak_2]$.
See Fact~\ref{fact:Ubf-description}.

We call $\Ibf_{\Hfrak}$ resp.~$\Dbf_{\Hfrak}$ the \emph{inertia} resp.~\emph{decomposition groups} associated to $\Hfrak$.
These subgroups of $\Pibf^{a}_{\Afrak}$ agree with the usual inertia resp.~decomposition groups associated to the prime Weil divisor $\Hfrak$.
The following is a well-known property of inertia/decomposition groups of prime divisors.
\begin{fact}\label{fact:decomp-alt-pair}
  Let $\Hfrak$ be a hyperplane in an affine space $\Afrak$, and let $\sigma \in \Ibf_{\Hfrak}$ and $\tau \in \Dbf_{\Hfrak}$ be given.
  Then $[\sigma,\tau] = 0$.
\end{fact}

We will also need a converse to this assertion.
\begin{lemma}\label{lemma:decomp-alt-pair-converse}
  Let $\Hfrak$ be a hyperplane in an affine space $\Afrak$, and $\sigma$ a generator of $\Ibf_{\Hfrak}$.
  Suppose $\tau \in \Pibf^{a}_{\Afrak}$ satisfies $[\sigma,\tau] = 0$.
  Then $\tau \in \Dbf_{\Hfrak}$.
\end{lemma}
\begin{proof}
  This follows from the explicit description of $\Ubf_{\Hfrak}^{1}$ along with Lemma~\ref{lemma:alt-pairs}.
  Indeed, note first that $\sigma[\Hfrak] \in \Lambda^{\times}$ and $\sigma [\Hfrak'] = 0$ for all $\Hfrak' \neq \Hfrak$, since $\sigma$ generates $\Ibf_{\Hfrak}$.
  If $(\Hfrak,\Hfrak_{1})$ is a parallel pair, then one has
  \[ \sigma[\Hfrak] \cdot \tau[\Hfrak_{1}] = \sigma [\Hfrak_{1}] \cdot \tau [\Hfrak] = 0 \]
  by Lemma~\ref{lemma:alt-pairs}, and as $\sigma[\Hfrak]$ is a unit it follows that $\tau[\Hfrak_{1}] = 0$.
  If $(\Hfrak,\Hfrak_{1},\Hfrak_{2})$ is a dependent triple then the same lemma implies similarly that
  \[ \sigma[\Hfrak] \cdot (\tau([\Hfrak_{1}] - [\Hfrak_{2}])) = (\sigma[\Hfrak_{1}] -\sigma[\Hfrak_{2}]) \cdot (\tau[\Hfrak] - \tau[\Hfrak_{2}]) = 0, \]
  hence $\tau([\Hfrak_{1}] - [\Hfrak_{2}]) = 0$.
  As $\Ubf^{1}_{\Hfrak}$ is generated by $[\Hfrak_{1}]$ for parallel pairs $(\Hfrak,\Hfrak_{1})$ and $[\Hfrak_{1}] - [\Hfrak_{2}]$ for dependent triples $(\Hfrak,\Hfrak_{1},\Hfrak_{2})$, we find that $\tau \in \Dbf_{\Hfrak}^{1}$ as contended.
\end{proof}

\subsection{The main local theorem}

Our task is to prove the following result which will be used to parameterize hyperplanes in the proof of the main theorem.
\begin{theorem}\label{theorem:local-theory}
  Let $\Afrak$ be an affine space of dimension $\geq 2$ with coordinates $\xbf = (x_{1},\ldots,x_{n})$.
  Let $\Ibf \subset \Dbf$ be submodules of $\Pibf^{a}_{\Afrak}$ and let $i \in \{1,\ldots,n\}$ be given, and let $\zbf$ be the partial system of coordinates obtained from $\xbf$ by deleting $x_{i}$.
  Then there exists a hyperplane $\Hfrak_{0}$ which dominates $\Abb^{n-1}_{\zbf}$ via $\pi_{\zbf}$, such that $\Ibf = \Ibf_{\Hfrak_{0}}$ and $\Dbf = \Dbf_{\Hfrak_{0}}$, if and only if there exists some $\sigma \in \Ibf$ with $\sigma \notin \ell \cdot \Pibf^{a}_{\Afrak}$ such that $\Ibf = \Lambda \cdot \sigma$,
  \[ \Dbf = \Ccalbf(\sigma) := \{\tau \in \Pibf^{a} \ | \ [\sigma,\tau] = 0\}, \]
  and the following additional conditions hold:
  \begin{enumerate}
    \item There exists some hyperplane $\Hfrak'$ which dominates $\Abb^{n-1}_{\zbf}$ via $\pi_{\zbf}$ such that $\sigma[\Hfrak'] = 0$.
    \item The element $\sigma$ maps trivially under $\pibf_{\zbf} : \Pibf^{a}_{\Afrak} \to \Pibf_{\zbf}^{a}$.
    \item The map $\Dbf / \Ibf \to \Pibf^{a}_{\zbf}$ induced by $\pibf_{\zbf}$ is an isomorphism.
  \end{enumerate}
\end{theorem}
The rest of this section will be devoted to proving this key result.

\subsection{The implication ``$\,\Rightarrow$''}

Suppose first that $i \in \{1,\ldots,n\}$ is given, $\zbf$ is as in the statement of the theorem, and that there exists a hyperplane $\Hfrak_{0}$ dominating $\Abb^{n-1}_{\zbf}$ such that $\Ibf_{\Hfrak_{0}} = \Ibf$ and $\Dbf_{\Hfrak_{0}} = \Dbf$.
Then $\Ibf = \Ibf_{\Hfrak_{0}}$ has a generator $\sigma$ satisfying $\sigma [\Hfrak_{0}]  = 1$ hence $\sigma \notin \ell \cdot \Pibf^{a}_{\Afrak}$.
Also $\sigma[\Hfrak] = 0$ for all hyperplanes $\Hfrak \neq \Hfrak_{0}$, by the explicit description of $\Ibf_{\Hfrak_{0}}$, and so condition~(1) clearly holds true for any hyperplane $\Hfrak'$ dominating $\Abb^{n-1}_{\zbf}$ which is distinct from $\Hfrak_{0}$.
Condition (2) follows from the fact that $\Hfrak_{0}$ dominates $\Abb^{n-1}_{\zbf}$.
As for condition (3), recall that $\Dbf_{\Hfrak_{0}}/\Ibf_{\Hfrak_{0}} \cong \Pibf^{a}_{\Hfrak_{0}}$, and the induced map $\Pibf^{a}_{\Hfrak_{0}} \to \Pibf^{a}_{\zbf}$ is simply the map obtained by functoriality (and taking limits) from the composition
\[ \Hfrak_{0} \hookrightarrow \Afrak \xrightarrow{\pi_{\zbf}} \Abb^{n-1}_{\zbf} \]
which is an isomorphism of varieties.
This map $\Dbf_{\Hfrak_{0}}/\Ibf_{\Hfrak_{0}} \to \Pibf^{a}_{\zbf}$ is thus also an isomorphism, so condition (3) follows.
Lastly, we have $\Dbf_{\Hfrak_{0}} = \Ccalbf(\sigma)$ by Fact~\ref{fact:decomp-alt-pair} and Lemma~\ref{lemma:decomp-alt-pair-converse}.

\subsection{The implication ``$\Leftarrow$''}

Suppose now that $\zbf$ and $\sigma$ are as in the statement of the theorem so that conditions (1), (2) and (3) hold true and $\Ibf = \Lambda \cdot \sigma$, $\Dbf = \Ccalbf(\sigma)$.
Since $\sigma$ is not a multiple of $\ell$ by assumption, there is a hyperplane $\Hfrak_{0}$ such that $\sigma[\Hfrak_{0}] \in \Lambda^{\times}$.
Condition (2) implies that $\Hfrak_{0}$ dominates $\Abb^{n-1}_{\zbf}$, for otherwise its image would be a hyperplane in $\Abb^{n-1}_{\zbf}$ and thus the image of $\sigma$ in $\Pibf^{a}_{\zbf}$ would have to be nontrivial.
Our goal is thus to show that $\Lambda \cdot \sigma = \Ibf_{\Hfrak_{0}}$ and $\Ccalbf(\sigma) = \Dbf_{\Hfrak_{0}}$.
We will use the explicit descriptions of $\Ibf_{\Hfrak_{0}}$ and $\Dbf_{\Hfrak_{0}}$ to do this.

Let us show that $\sigma[\Hfrak] = 0$ for all hyperplanes $\Hfrak \neq \Hfrak_{0}$.
Suppose first that $\Hfrak_{1}$ is a hyperplane such that $(\Hfrak_{0},\Hfrak_{1})$ form a parallel pair.
Choose a distinct hyperplane $\Hfrak_{2}$ which dominates $\Abb^{n-1}_{\zbf}$ and which meets $\Hfrak_{0}$.
Note that $\Hfrak_{2}$ necessarily meets $\Hfrak_{1}$ as well.
There exist two distinct hyperplanes which are vertical over $\Abb^{n-1}_{\zbf}$ say $\Hfrak_{1}'$, $\Hfrak_{2}'$, such that $(\Hfrak_{0},\Hfrak_{2},\Hfrak_{1}')$ and $(\Hfrak_{1},\Hfrak_{2},\Hfrak_{2}')$ are dependent triples, see Figure~\ref{figure:lines} for a picture in the two-dimensional case.
Explicitly, $\Hfrak_{1}' = \pi_{\zbf}^{-1}(\pi_{\zbf}(\Hfrak_{0} \cap \Hfrak_{2}))$ and $\Hfrak_{2}' = \pi_{\zbf}^{-1}(\pi_{\zbf}(\Hfrak_{1} \cap \Hfrak_{2}))$.
\begin{figure}
\begin{tikzpicture}
  \draw (0,-1) -- (5,-1) node[anchor=west] {\tiny $\Abb^{n-1}_{\zbf}$};
  %\draw (0,0) rectangle (5,4);
  \draw (0,1) -- (5,1.5) node[anchor=west] {\tiny $\Hfrak_{0}$};
  \draw (0,2) -- (5,2.5) node[anchor=west] {\tiny $\Hfrak_{1}$};
  \draw (0,3) node[anchor=east] {\tiny $\Hfrak_{2}$} -- (5,1) ;
  \draw (2,0) node[anchor=north] {\tiny $\Hfrak_{2}'$} -- (2,4) ;
  \draw (4,0) node[anchor=north] {\tiny $\Hfrak_{1}'$} -- (4,4);
\end{tikzpicture}
\caption{Configuration for proof of Theorem~\ref{theorem:local-theory}.}\label{figure:lines}
\end{figure}
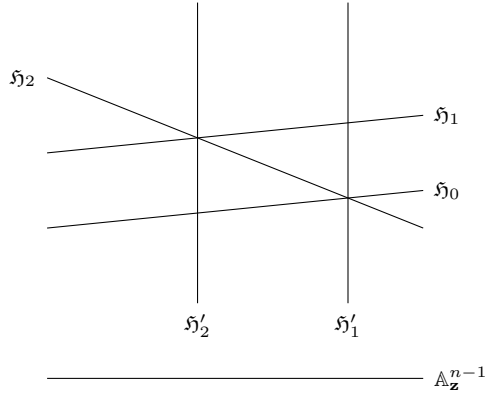

By condition (3), there exists some $\tau \in \Ccalbf(\sigma)$ such that $\tau[\Hfrak_{1}'] = 0$ and $\tau[\Hfrak_{2}'] = 1$.
And by condition (2), we have $\sigma[\Hfrak_{1}'] = \sigma[\Hfrak_{2}'] = 0$.
Put $\alpha_{i} := \sigma[\Hfrak_{i}]$ and $\beta_{i} := \tau[\Hfrak_{i}]$.
Since $(\Hfrak_{0},\Hfrak_{1})$ is a parallel pair and $[\sigma,\tau] = 0$, by Lemma~\ref{lemma:alt-pairs}, we have
\begin{equation}
\alpha_{0} \cdot \beta_{1} = \alpha_{1} \cdot \beta_{0}.
\end{equation}
Also, since $(\Hfrak_{0},\Hfrak_{2},\Hfrak_{1}')$ is a dependent triple with $\tau[\Hfrak_{1}'] = 0$ and $\sigma[\Hfrak_{1}'] = 0$, we have again from Lemma~\ref{lemma:alt-pairs} that
\begin{equation}
\alpha_{0} \cdot \beta_{2} = \alpha_{2} \cdot \beta_{0}.
\end{equation}
Recall that $\alpha_{0} \in \Lambda^{\times}$, so we may solve for $\beta_{1},\beta_{2}$ as follows:
\begin{equation}\label{eqn:alt-pair-calc}
\beta_{1} = \frac{\alpha_{1} \cdot \beta_{0}}{\alpha_{0}}, \ \beta_{2} = \frac{\alpha_{2} \cdot \beta_{0}}{\alpha_{0}}.
\end{equation}

Next, since $(\Hfrak_{1},\Hfrak_{2},\Hfrak_{2}')$ is a dependent triple, with $\tau[\Hfrak_{2}'] = 1$ and $\sigma[\Hfrak_{2}'] = 0$, we have from Lemma~\ref{lemma:alt-pairs} that
\begin{equation}
  \alpha_{1} \cdot (\beta_{2} - 1) = \alpha_{2} \cdot (\beta_{1} - 1).
\end{equation}
Substituting equation~\eqref{eqn:alt-pair-calc}, we find
\begin{equation}
  \alpha_{1} \cdot \left(\frac{\alpha_{2} \cdot \beta_{0}}{\alpha_{0}} - 1\right) = \alpha_{2} \cdot \left(\frac{\alpha_{1} \cdot \beta_{0}}{\alpha_{0}} - 1\right).
\end{equation}
Subtracting a term of the form $(\alpha_{1} \cdot \alpha_{2} \cdot \beta_{0})/\alpha_{0}$ from both sides and negating, we find that $\alpha_{1} = \alpha_{2}$.
To summarize, we have obtained the following.
\begin{lemma}\label{lemma:local-theory-step}
  In the above context, suppose that $(\Hfrak_{0},\Hfrak_{1})$ is a parallel pair, and that $\Hfrak_{2}$ is a distinct hyperplane which dominates $\Abb^{n-1}_{\zbf}$ and which meets $\Hfrak_{0}$.
  Then $\sigma[\Hfrak_{1}] = \sigma[\Hfrak_{2}]$.
\end{lemma}

Let $\Hfrak'$ be as in condition (1), so that $\Hfrak'$ dominates $\Abb^{n-1}_{\zbf}$ and $\sigma[\Hfrak'] = 0$.
As $\sigma[\Hfrak_{0}] \ne 0$, we have $\Hfrak_{0} \ne \Hfrak'$.
If $(\Hfrak_{0},\Hfrak')$ is not a parallel pair, then we choose a distinct hyperplane $\Hfrak_{1}$ such that $(\Hfrak_{0},\Hfrak_{1})$ is a parallel pair, while Lemma~\ref{lemma:local-theory-step} ensures that $\sigma[\Hfrak_{1}] = \sigma[\Hfrak'] = 0$.
In any case, we deduce the following.
\begin{lemma}\label{lemma:local-theory-step-parallel}
  There exists a parallel pair of the form $(\Hfrak_{0},\Hfrak_{1})$ such that $\sigma[\Hfrak_{1}] = 0$.
\end{lemma}

Let $\Hfrak_{1}$ be as in Lemma~\ref{lemma:local-theory-step-parallel}.
Now suppose that $\Hfrak_{2}$ is any hyperplane distinct from $\Hfrak_{0}$.
We wish to show that $\sigma[\Hfrak_{2}] = 0$.
If $\Hfrak_{2} = \Hfrak_{1}$ then we are done, and if $\Hfrak_{2}$ is vertical over $\Abb^{n-1}_{\zbf}$ then we are done by condition (2).
If $(\Hfrak_{0},\Hfrak_{2})$ is a parallel pair and $\Hfrak_{2} \ne \Hfrak_{1}$, then we may choose a distinct hyperplane $\Hfrak_{3}$ which dominates $\Abb^{n-1}_{\zbf}$ and which meets $\Hfrak_{0}$, hence it also meets $\Hfrak_{1}$ and $\Hfrak_{2}$.
Applying Lemma~\ref{lemma:local-theory-step}, we see that
\[ \sigma[\Hfrak_{2}] = \sigma[\Hfrak_{3}] = \sigma[\Hfrak_{1}] = 0. \]
Finally, if $\Hfrak_{2}$ meets $\Hfrak_{0}$ and is not vertical over $\Abb^{n-1}_{\zbf}$, then it must dominate $\Abb^{n-1}_{\zbf}$, and thus by Lemma~\ref{lemma:local-theory-step} we have $\sigma[\Hfrak_{2}] = \sigma[\Hfrak_{1}] = 0$.
This indeed shows that $\sigma[\Hfrak_{2}] = 0$ whenever $\Hfrak_{2}\neq \Hfrak_{0}$, as contended.
By the explicit description of $\Ibf_{\Hfrak_{0}}$, it follows that $\Lambda \cdot \sigma = \Ibf_{\Hfrak_{0}}$, while Fact~\ref{fact:decomp-alt-pair} and Lemma~\ref{lemma:decomp-alt-pair-converse} imply that $\Ccalbf(\sigma) = \Dbf_{\Hfrak_{0}}$.
This concludes the proof of Theorem~\ref{theorem:local-theory}.

\section{The proof of the main theorem}

We now consider the context of Theorem~\ref{maintheorem:general}.
Put $\Gfrak := \Aut^{c}(\Pi^{a}_{\Hscr_{\xbf,S}})$.
We start by constructing actions of $\Gfrak$ on various objects, all of which have a natural $\Gal(k|k_{0})$ action.
As always, when working with profinite groups we consider them as objects in the category with \emph{outer morphisms}.

\subsection{The action on fundamental groups}\label{subsec:action-fund-group}

Merely by definition, $\Gfrak$ acts naturally on $\Pi^{a}_{X}$ for all objects $X$ in $\Hscr_{\xbf,S}$, in a way compatible with all of the morphisms arising from morphisms in $\Hscr_{\xbf,S}$ and with the relation $[-,-] = 0$ on $\Pi^{a}_{X}$ in the sense that for $\gamma \in \Gfrak$ and $\sigma,\tau \in \Pi^{a}_{X}$, one has $[\sigma,\tau] = 0$ if and only if $[\gamma\sigma,\gamma\tau] = 0$.
In particular, $\Gfrak$ acts naturally on $\Pibf^{a}_{\Afrak}$, compatibly with the relation $[-,-] = 0$, and the canonical maps
\[ \pi_{z} : \Pibf^{a}_{\Afrak} \to \Pi^{a}_{\Pbb^{1} \smin \{0,1,\infty\}}, \]
are $\Gfrak$-equivariant for all $z$ of the form $x_{i} - s$ where $1 \le i \le n$ and $s \in S$, or of the form $x_{i} - x_{j}$ where $1 \le i < j \le n$.

For every object $X$ of $\Hscr_{\xbf,S}$, we may also define an action of $\Gfrak$ on $\HH^{1}(X)$ by identifying it with $\Hom(\Pi^{a}_{X},\Lambda)$ via~\eqref{eqn:H1-pairing}.
Note that $\Gfrak$ acts \emph{trivially} on $\Lambda$ in this formula.

Passing to the colimit, we obtain a natural action of $\Gfrak$ on $\HHbf^{1}_{\Afrak}$, and the maps
\[ \iota_{z} : \HH^{1}(\Pbb^{1} \smin \{0,1,\infty\}) \to \HHbf^{1}_{\Afrak} \]
are $\Gfrak$-equivariant as well for the $z$ as above.

\subsection{Linear projections}

The goal of this subsection is to prove the following technical proposition.
\begin{proposition}\label{proposition:linear-proj}
  Let $\zbf$ be any partial system of coordinates contained in $\xbf$.
  The action of $\Gfrak$ on $\HHbf^{1}_{\Afrak}$ restricts to an action on the image of the embedding $\iotabf_{\zbf} : \HHbf^{1}_{\zbf} \hookrightarrow \HHbf^{1}_{\Afrak}$.
\end{proposition}
\begin{proof}
  We proceed by induction on the length of $\zbf$ with the base-case where $\zbf$ is empty is vacuously true.
  Write $\zbf = (z_{1},\ldots,z_{m})$ and $\wbf = (z_{1},\ldots,z_{m},w)$ with both $\zbf$ and $\wbf$ being contained in $\xbf$, and assume the result holds true for $\zbf$.
  We will characterize the image of $\iotabf_{\wbf} : \HHbf^{1}_{\wbf} \to \HHbf^{1}_{\Afrak}$ using the following data:
  \begin{enumerate}
    \item The image of $\iotabf_{\zbf} : \HHbf^{1}_{\zbf} \to \HHbf^{1}_{\Afrak}$.
    \item The image of $\iota_{w} : \HH^{1}(\Pbb^{1} \smin \{0,1,\infty\}) \to \HHbf^{1}_{\Afrak}$.
    \item The pairing $\Pibf^{a}_{\Afrak} \times \HHbf^{1}_{\Afrak} \to \Lambda$.
    \item The relation $[-,-] = 0$ on $\Pibf^{a}_{\Afrak}$.
  \end{enumerate}
  This suffices since the action $\Gfrak$ is compatible with items (2), (3) and (4) by definition, and is compatible with (1) by the inductive hypothesis.
  Precisely, we claim that $\alpha \in \HHbf^{1}_{\Afrak}$ is in the image of $\iotabf_{\wbf} : \HHbf^{1}_{\wbf} \to \HHbf^{1}_{\Afrak}$ if and only if it satisfies the following condition.
  \begin{condition}\label{condition}
    For all $\beta$ in the image of $\iota_{w} : \HH^{1}(\Pbb^{1} \smin \{0,1,\infty\}) \to \HHbf^{1}_{\Afrak}$ and for all $\sigma,\tau \in \Pibf^{a}_{\Afrak}$ which pair trivially with the image of $\iotabf_{\zbf} : \HHbf^{1}_{\zbf} \to \HHbf^{1}_{\Afrak}$ and which satisfy $[\sigma,\tau] = 0$, one has $\sigma \alpha \cdot \tau \beta = \sigma \beta \cdot \tau \alpha$.
  \end{condition}

  First let us show that any element in the image of $\iotabf_{\wbf} : \HHbf^{1}_{\wbf} \to \HHbf^{1}_{\Afrak}$ satisfies this condition.
  As this condition is clearly linear in $\alpha$, it suffices to show that the Kummer classes $[\Hfrak]$, with $\Hfrak$ a hyperplane vertical over $\Abb^{m+1}_{\wbf}$, satisfy this condition.
  So let $\Hfrak$ be such a hyperplane, and let $\beta$, $\sigma$ and $\tau$ be as in the condition.
  Write $\beta = a \cdot [\Hfrak_{w}] + b \cdot [\Hfrak_{1-w}]$ with $a,b \in \Lambda$.
  If $\Hfrak$ is vertical over $\Abb^{m}_{\zbf}$, then the assertion is clear as one has $\sigma [\Hfrak] = \tau [\Hfrak] = 0$.
  Otherwise, $\Hfrak$ dominates $\Abb^{m}_{\zbf}$, hence there exist two hyperplanes $\Hfrak_{1}$, $\Hfrak_{2}$ which are vertical over $\Abb^{m}_{\zbf}$ such that $(\Hfrak,\Hfrak_{w},\Hfrak_{1})$ and $(\Hfrak,\Hfrak_{1-w},\Hfrak_{2})$ are dependent triples.
  As $\sigma [\Hfrak_{i}] = \tau [\Hfrak_{i}] = 0$, the assertion follows from Lemma~\ref{lemma:alt-pairs} as follows:
  \begin{align*}
    \sigma [\Hfrak] \cdot \tau \beta &= \sigma[\Hfrak] \cdot (a \cdot \tau [\Hfrak_{w}] + b \cdot \tau [\Hfrak_{1-w}]) \\
                                     &= a \cdot \sigma [\Hfrak] \cdot \tau [\Hfrak_{w}] + b \cdot \sigma[\Hfrak] \cdot \tau [\Hfrak_{1-w}] \\
                                     &= a \cdot (\sigma [\Hfrak] - \sigma [\Hfrak_{1}]) \cdot (\tau [\Hfrak_{w}] - \tau [\Hfrak_{1}]) + b \cdot (\sigma [\Hfrak] - \sigma [\Hfrak_{2}]) \cdot (\tau [\Hfrak_{1-w}] - \tau [\Hfrak_{2}]) \\
                                     &= a \cdot (\sigma [\Hfrak_{w}] - \sigma [\Hfrak_{1}]) \cdot (\tau [\Hfrak] - \tau [\Hfrak_{1}]) + b \cdot (\sigma [\Hfrak_{1-w}] - \sigma [\Hfrak_{2}]) \cdot (\tau [\Hfrak] - \tau [\Hfrak_{2}]) \\
                                     &= a \cdot \sigma [\Hfrak_{w}] \cdot \tau [\Hfrak] + b \cdot \sigma[\Hfrak_{1-w}] \cdot \tau [\Hfrak] \\
                                     &= \sigma \beta \cdot \tau [\Hfrak].
  \end{align*}

  As for the converse, let $\alpha \in \HHbf^{1}_{\Afrak}$ be given and write $\alpha$ as a linear combination of the form
  \[ \alpha = \sum_{\Hfrak} c_{\Hfrak} \cdot [\Hfrak] \]
  where $\Hfrak$ varies over all hyperplanes in $\Afrak$ and $c_{\Hfrak} \in \Lambda$.
  Assume $\alpha$ satisfies Condition~\ref{condition}.
  We must show that $c_{\Hfrak} = 0$ for any $\Hfrak$ which is \emph{not} vertical over $\Abb^{m+1}_{\wbf}$.
  Let $\Hfrak$ be such a hyperplane, and take $\beta = [\Hfrak_{w}]$.
  Let $\sigma$ be the unique element of $\Pibf^{a}_{\Afrak}$ acting on Kummer classes of hyperplanes as the indicator function of $[\Hfrak]$, and note that $\sigma$ generates $\Ibf_{\Hfrak}$.
  Note that $\Hfrak \cap \Hfrak_{w}$ is a hyperplane in $\Hfrak$ which is not of the form $\Hfrak \cap \Hfrak'$ for any hyperplane $\Hfrak'$ which is vertical over $\Abb^{m}_{\zbf}$.
  Let $\bar\tau \in \Pibf^{a}_{\Hfrak}$ be an element which satisfies $\bar\tau [\Hfrak \cap \Hfrak_{w}] = 1$ and $\bar\tau [\Hfrak \cap \Hfrak'] = 0$ for all $\Hfrak'$ vertical over $\Abb^{m}_{\zbf}$.
  Choose $\tau \in \Dbf_{\Hfrak}$ mapping to $\bar\tau$ under the projection $\Dbf_{\Hfrak} \twoheadrightarrow \Pibf^{a}_{\Hfrak}$.
  Recall that $[\sigma,\tau] = 0$ by Fact~\ref{fact:decomp-alt-pair}, while both $\sigma$ and $\tau$ pair trivially with the image of $\iotabf_{\zbf} : \HHbf^{1}_{\zbf} \to \HHbf^{1}_{\Afrak}$.
  Condition~\ref{condition} can therefore be applied with this $\beta$, $\sigma$ and $\tau$, and the equation appearing in this condition reduces to $c_{\Hfrak} = 0$, as required.
\end{proof}

Proposition~\ref{proposition:linear-proj} shows that the image of $\iotabf_{\zbf} : \HHbf^{1}_{\zbf} \hookrightarrow \HHbf^{1}_{\Afrak}$ is invariant under the action of $\Gfrak$ for any partial system of coordinates $\zbf$ contained in $\xbf$.
We define an action of $\Gfrak$ on $\HHbf^{1}_{\zbf}$ via this inclusion, and this allows us to define an action of $\Gfrak$ on $\Pibf^{a}_{\zbf}$ by identifying it with $\Hom(\HHbf^{1}_{\zbf},\Lambda)$ via~\eqref{eqn:univ-H1-pairing}.
The canonical projection $\pibf_{\zbf} : \Pibf^{a}_{\Afrak} \to \Pibf^{a}_{\zbf}$ is then $\Gfrak$-equivariant by definition.

\subsection{The action on hyperplanes}

We now construct an action of $\Gfrak$ on the hyperplanes in $\Afrak$ which is compatible with decomposition theory.
For the rest of the proof, we write $\Planes := \Planes_{\Afrak}$.
\begin{proposition}
   There is an action of $\Gfrak$ on $\Planes$ which is uniquely determined by the condition that for $\gamma \in \Gfrak$ and a hyperplane $\Hfrak$, one has $\gamma \Ibf_{\Hfrak} = \Ibf_{\gamma \Hfrak}$, and $\gamma \Dbf_{\Hfrak} = \Dbf_{\gamma \Hfrak}$.
\end{proposition}
\begin{proof}
  Let $\Hfrak$ be a hyperplane in $\Afrak$ and $\gamma \in \Gfrak$ be given.
  Since two hyperplanes are equal if and only if their inertia groups are equal, it suffices to show that there exists a hyperplane $\Hfrak_{0}$ such that $\gamma \Ibf_{\Hfrak} = \Ibf_{\Hfrak_{0}}$ and $\gamma \Dbf_{\Hfrak} = \Dbf_{\Hfrak_{0}}$.

  We will use Theorem~\ref{theorem:local-theory} to do this.
  Let $\sigma' \in \Ibf_{\Hfrak}$ be a generator and put $\sigma := \gamma \sigma'$.
  Let $\zbf$ be a partial system of coordinates obtained by deleting one of the $x_{i}$ such that $\Hfrak$ dominates $\Abb^{n-1}_{\zbf}$.
  By Theorem~\ref{theorem:local-theory}, we have $\Dbf_{\Hfrak} = \Ccalbf(\sigma')$, and the following conditions hold:
  \begin{enumerate}
    \item For every hyperplane $\Hfrak' \neq \Hfrak$, one has $\sigma'[\Hfrak'] = 0$.
    \item The element $\sigma'$ maps trivially under $\pibf_{\zbf} : \Pibf^{a}_{\Afrak} \to \Pibf^{a}_{\zbf}$.
    \item The map $\Ccalbf(\sigma')/\Lambda \cdot \sigma' \to \Pibf^{a}_{\zbf}$ induced by $\pibf_{\zbf}$ is an isomorphism.
  \end{enumerate}
  Since the action of $\gamma$ is compatible with $[-,-] = 0$, we have $\gamma \Dbf_{\Hfrak} = \Ccalbf(\sigma)$.
  Also, since $\pibf_{\zbf} : \Pibf^{a}_{\Afrak} \to \Pibf^{a}_{\zbf}$ is $\Gfrak$-equivariant as noted above, we see that $\sigma$ maps trivially under $\pibf_{\zbf}$ and that the induced map $\Ccalbf(\sigma)/\Lambda \cdot \sigma \to \Pibf^{a}_{\zbf}$ is an isomorphism.

  To conclude using Theorem~\ref{theorem:local-theory}, we must therefore show that there exists some hyperplane $\Hfrak'$ dominating $\Abb^{n-1}_{\zbf}$ such that $\sigma[\Hfrak'] = 0$.
  If $\Hfrak$ has the form $\Hfrak_{x_{j}}$ or $\Hfrak_{x_{j}-1}$, then $\sigma'$ pairs trivially with both $[\Hfrak_{x_{r}-x_{s}}]$ and $[\Hfrak_{x_{r}-x_{s}-1}]$ for every $1 \le r < s \le n$, hence the same is true for $\sigma$ since
  \[ \iota_{x_{r}-x_{s}} : \HH^{1}(\Pbb^{1} \smin \{0,1,\infty\}) \to \HHbf^{1}_{\Afrak} \]
  is $\gamma$-equivariant with image $\langle [\Hfrak_{x_{r}-x_{s}}],[\Hfrak_{x_{r}-x_{s}-1}] \rangle$.
  Otherwise $\sigma'$ pairs trivially with $[\Hfrak_{x_{j}}]$, $[\Hfrak_{x_{j}-1}]$, so the same argument shows that $\sigma$ pairs trivially with $[\Hfrak_{x_{j}}]$ and $[\Hfrak_{x_{j}-1}]$ as
  \[ \iota_{x_{j}} : \HH^{1}(\Pbb^{1} \smin \{0,1,\infty\}) \to \HHbf^{1}_{\Afrak} \]
  is also $\gamma$-equivariant.
  In any case, there exists some hyperplane $\Hfrak'$ dominating $\Abb^{n-1}_{\zbf}$ such that $\sigma[\Hfrak'] = 0$, and this concludes the proof using Theorem~\ref{theorem:local-theory}.
\end{proof}

Recall that the pairing~\eqref{eqn:univ-H1-pairing} is $\Gfrak$-equivariant, where $\Gfrak$ acts trivially on $\Lambda$, while the orthogonal of $\Ibf_{\Hfrak}$ resp.~$\Dbf_{\Hfrak}$ with respect to this pairing is $\Ubf_{\Hfrak}$ resp.~$\Ubf_{\Hfrak}^{1}$.
We therefore deduce the following lemma as well.

\begin{lemma}\label{lemma:action-Ubf}
  Let $\gamma \in \Gfrak$, $\Hfrak \in \Planes$ be given.
  One has $\gamma \Ubf_{\Hfrak} = \Ubf_{\gamma \Hfrak}$ and $\gamma \Ubf_{\Hfrak}^{1} = \Ubf_{\gamma \Hfrak}^{1}$.
\end{lemma}

We will also need the following variant involving Kummer classes.
\begin{lemma}\label{lemma:action-Kummer-classes}
  Let $\gamma \in \Gfrak$ and $\Hfrak \in \Planes$ be given.
  Then one has $\Lambda \cdot \gamma[\Hfrak] = \Lambda \cdot [\gamma \Hfrak]$.
\end{lemma}
\begin{proof}
  Note that
  \[ \Lambda \cdot [\Hfrak] = \bigcap_{\Hfrak' \neq \Hfrak} \Ubf_{\Hfrak'}. \]
  The lemma therefore follows from Lemma~\ref{lemma:action-Ubf}.
\end{proof}

And finally, we will need the following lemma showing compatibility with parallel pairs and dependent triples.
\begin{lemma}\label{lemma:colin-helper}
  Let $\gamma \in \Gfrak$ and $\Hfrak_{1},\Hfrak_{2},\Hfrak_{3} \in \Planes$ be given.
  The following hold:
  \begin{enumerate}
    \item If $(\Hfrak_{1},\Hfrak_{2})$ is a parallel pair then so is $(\gamma \Hfrak_{1},\gamma \Hfrak_{2})$.
    \item If $(\Hfrak_{1},\Hfrak_{2},\Hfrak_{3})$ is a dependent triple then so is $(\gamma \Hfrak_{1},\gamma \Hfrak_{2},\gamma \Hfrak_{3})$.
  \end{enumerate}
\end{lemma}
\begin{proof}
  Both assertions follow from Lemmas~\ref{lemma:action-Ubf} and~\ref{lemma:action-Kummer-classes}, the explicit descriptions of $\Ubf_{\Hfrak}$ and $\Ubf_{\Hfrak}^{1}$, and the explicit description of $\sfrak_{\Hfrak}$, using the following characterizations.
  First, a pair of distinct hyperplanes $(\Hfrak_{1},\Hfrak_{2})$ is a parallel pair if and only of $\Lambda \cdot [\Hfrak_{1}] \subset \Ubf^{1}_{\Hfrak_{2}}$.
  Second, a triple of distinct hyperplanes $(\Hfrak_{1},\Hfrak_{2},\Hfrak_{3})$ is a dependent triple if and only if the images of $\Lambda \cdot [\Hfrak_{1}]$ and of $\Lambda \cdot [\Hfrak_{2}]$ are nontrivial and agree in the quotient $\Ubf_{\Hfrak_{3}}/\Ubf_{\Hfrak_{3}}^{1}$.
\end{proof}

\subsection{The dual projective space}

Given $\Hfrak \in \Planes$, let $\pfrak(\Hfrak) = (a_{0}:\cdots:a_{n}) \in \Pbb^{n}(k)$ denote the unique point, written in homogeneous coordinates, where $\Hfrak$ is the zero-locus of $a_{0} + a_{1} \cdot x_{1} + \cdots + a_{n} \cdot x_{n}$.
Note that $\pfrak : \Planes \to \Pbb^{n}(k)$ is a well-defined embedding whose image consists of every point of $\Pbb^{n}(k)$, except for the point $\pfrak_{0} := (1:0:\cdots:0)$.
We define an action of $\Gfrak$ on $\Pbb^{n}(k)$ in such a way so that $\gamma \in \Gfrak$ acts as $\gamma \pfrak_{0} = \pfrak_{0}$ and $\gamma \pfrak(\Hfrak) = \pfrak(\gamma \Hfrak)$ for every $\Hfrak \in \Planes$.

In other words, if we identify $\Afrak$ with $\Abb^{n}_{\xbf}$ via our choice of coordinates, and further embed $\Abb^{n}_{\xbf}$ in $\Pbb^{n}_{\xbf}$ in the usual way, then we are identifying the hyperplanes in $\Pbb^{n}_{\xbf}$ with the points in the \emph{dual projective space}, and $\pfrak_{0}$ corresponds to the hyperplane at infinity.
Note that the lines in this dual projective space correspond to pencils of hyperplanes in $\Pbb^{n}_{\xbf}$.
The following proposition shows that the action of $\Gfrak$ on this dual projective space respects such lines.
Recall that a bijection $\gamma : \Pbb^{n}(k) \to \Pbb^{n}(k)$ is called a \emph{collineation} provided that for all projective lines $\lfrak$ in $\Pbb^{n}(k)$, the image $\gamma(\lfrak)$ is another projective line.

\begin{proposition}\label{proposition:colineations}
  The action of $\Gfrak$ on $\Pbb^{n}(k)$ defined above acts by collineations.
\end{proposition}
\begin{proof}
  This follows directly from Lemma~\ref{lemma:colin-helper}, as follows.
  A triple of points of the form $(\pfrak_{0},\pfrak(\Hfrak_{1}),\pfrak(\Hfrak_{2}))$ is colinear in $\Pbb^{n}(k)$ if and only if $(\Hfrak_{1},\Hfrak_{2})$ is a parallel pair.
  Thus elements of $\Gfrak$ send any line passing through $\pfrak_{0}$ to another line passing through $\pfrak_{0}$.
  On the other hand, a triple of points of the form $(\pfrak(\Hfrak_{1}),\pfrak(\Hfrak_{2}),\pfrak(\Hfrak_{3}))$, no two of which are colinear with $\pfrak_{0}$, are colinear in $\Pbb^{n}(k)$ if and only if $(\Hfrak_{1},\Hfrak_{2},\Hfrak_{3})$ is a dependent triple.
  Thus, elements of $\Gfrak$ send any line not passing through $\pfrak_{0}$ to another line which does not pass through $\pfrak_{0}$.
\end{proof}

\subsection{Rigidification}

We now prove the following proposition which will allow us to \emph{rigidify} the action of $\Gfrak$ on $\Pbb^{n}(k)$.
\begin{proposition}\label{proposition:rigidification}
  Let $(z,c) \in \{x_{1},\ldots,x_{n}\} \times S$ and $1 \le i < j \le n$ be given.
  The hyperplanes $\Hfrak_{z-c}$, $\Hfrak_{z-c-1}$, $\Hfrak_{x_{i}-x_{j}}$ and $\Hfrak_{x_{i}-x_{j}-1}$ are all fixed by the action of $\Gfrak$.
\end{proposition}
\begin{proof}
  For $\gamma \in \Gfrak$ and $\Hfrak \in \Planes$, note that $\gamma \Hfrak$ is the unique hyperplane such that $\gamma \Ibf_{\Hfrak}$ pairs nontrivially with $\Lambda \cdot [\gamma \Hfrak]$, while $\Lambda \cdot [\gamma \Hfrak] = \Lambda \cdot \gamma [\Hfrak]$ by Lemma~\ref{lemma:action-Kummer-classes}.
  For any $c \in S$ and any $1 \le i < j \le n$, the maps $\iota_{w} : \HH^{1}(\Pbb^{1} \smin \{0,1,\infty\}) \to \HHbf^{1}_{\Afrak}$ are $\Gfrak$-equivariant for $w \in \{x_{1}-c, \ldots, x_{n}-c, x_{i}-x_{j}\}$, with image spanned by $[\Hfrak_{w}]$ and $[\Hfrak_{w-1}]$.
  Hence $\gamma$ acts as a permutation of $\{\Hfrak_{w},\Hfrak_{w-1}\}$ for all such $w$.

  Note that for every $c \in S$ and $1 \le i < j \le n$, the following are all dependent triples:
  \[ (\Hfrak_{x_{i}-s},\Hfrak_{x_{j}-s},\Hfrak_{x_{i}-x_{j}}), (\Hfrak_{x_{i}-s-1},\Hfrak_{x_{j}-s-1},\Hfrak_{x_{i}-x_{j}}), \ (\Hfrak_{x_{i}-s-1},\Hfrak_{x_{j}-s},\Hfrak_{x_{i}-x_{j}-1}), \]
  while $(\Hfrak_{x_{i}-s},\Hfrak_{x_{j}-s-1},\Hfrak_{x_{i}-x_{j}-1})$ is \emph{not} a dependent triple.
  See Figure~\ref{figure:triples}.
  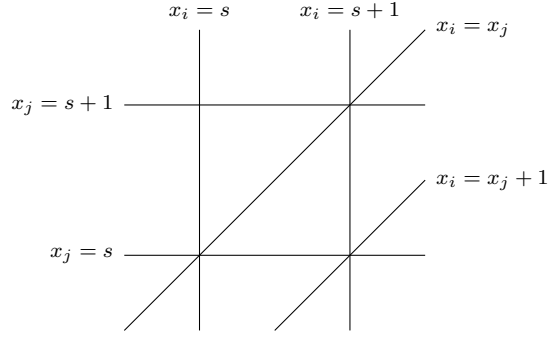
\begin{figure}
    \begin{tikzpicture}
      \draw (1,0) -- (1,4) node[anchor=south] {\tiny $x_{i} = s$};
      \draw (3,0) -- (3,4) node[anchor=south] {\tiny $x_{i} = s + 1$};
      \draw (0,1) node[anchor=east] {\tiny $x_{j} = s$} -- (4,1);
      \draw (0,3) node[anchor=east] {\tiny $x_{j} = s + 1$} -- (4,3);
      \draw (0,0) -- (4,4) node[anchor=west] {\tiny $x_{i} = x_{j}$};
      \draw (2,0) -- (4,2) node[anchor=west] {\tiny $x_{i} = x_{j} + 1$};
    \end{tikzpicture}
    \caption{Dependent triples in proof of Proposition~\ref{proposition:rigidification}}\label{figure:triples}
  \end{figure}
  The assertions all follow from these observations and Lemma~\ref{lemma:colin-helper}.
\end{proof}

\subsection{Compatibility with the field structure}

Recall that the \emph{general semilinear group} of $k^{n+1}$, denoted $\GSL(k^{n+1})$, is the group of semilinear automorphisms of $k^{n+1}$ as a $k$-module.
There is a natural map $k^{\times} \to \GSL(k^{n+1})$ sending $x \in k^{\times}$ to the associated homothety, with normal image, and we denote the quotient by the image of this map by $\PGSL(k^{n+1})$.

In light of Proposition~\ref{proposition:colineations}, the \emph{fundamental theorem of projective geometry}~\cite{artin} applied to this context be rephared as the assertion that that there exists a unique morphism
\[ \eta : \Gfrak \to \PGSL(k^{n+1}), \]
such that for all $\cbf = (c_{0},\ldots,c_{n}) \in k^{n+1}$ and all $\gamma \in \Gfrak$, the element $\eta(\gamma)(\cbf)$ represents $\gamma (c_{0} : \cdots : c_{n})$.
We will show that $\eta$ factors through the obvious inclusion $\Gal(k|k_{0}) \hookrightarrow \PGSL(k^{n+1})$.
Fix an element $\gamma \in \Gfrak$ for the rest of this subsection and let $\Gamma \in \GSL(k^{n+1})$ be a representative of $\eta(\gamma)$.
We also write $e_{0},\ldots,e_{n}$ for the standard basis of the $k$-module $k^{n+1}$.

\begin{lemma}\label{lemma:Gamma-epsilon}
  For all $i = 0,1,\ldots,n$, there exists some $\epsilon_{i} \in k^{\times}$ such that $\Gamma e_{i} = \epsilon_{i} \cdot e_{i}$.
\end{lemma}
\begin{proof}
  Note that $e_{0}$ represents $\pfrak_{0}$ while, for $i > 0$, the element $e_{i}$ represents $\pfrak(\Hfrak_{x_{i}})$.
  The assertion follows from the fact that $\gamma$ fixes $\pfrak_{0}$ and $\pfrak(\Hfrak_{x_{i}})$, which follows from the definition and from Proposition~\ref{proposition:rigidification}.
\end{proof}

\begin{lemma}\label{lemma:epsilon-eq}
  Let $\epsilon_{i}$ be as in Lemma~\ref{lemma:Gamma-epsilon}.
  Then one has $\epsilon_{0} = \epsilon_{1} = \cdots = \epsilon_{n}$.
\end{lemma}
\begin{proof}
  Let $0 < i \le n$ be given.
  Since $e_{0} - e_{i}$ represents $\pfrak(\Hfrak_{x_{i}-1})$, Proposition~\ref{proposition:rigidification} ensures that exists some $\epsilon$ such that $\Gamma(e_{0} - e_{i}) = \epsilon \cdot e_{0} - \epsilon \cdot e_{i}$.
  But $\Gamma$ is additive so $\Gamma(e_{0} - e_{i}) = \epsilon_{0} \cdot e_{0} - \epsilon_{1} \cdot e_{i}$, while $e_{0}$ is linearly independent from $e_{i}$, hence $\epsilon_{0} = \epsilon_{i}$.
\end{proof}

At this point we deduce from Lemmas~\ref{lemma:Gamma-epsilon} and~\ref{lemma:epsilon-eq} that $\eta$ factors through the obvious inclusion $\Aut(k) \hookrightarrow \PGSL(k^{n+1})$.
Replace $\eta$ with the corresponding map $\eta : \Gfrak \to \Aut(k)$, and replace $\Gamma$ with the element of $\GSL(k^{n+1})$ obtained as the image of $\eta(\gamma) \in \Aut(k)$.
This means that
\[ \Gamma(a_{0} \cdot e_{0} + \cdots + a_{n} \cdot e_{n}) = \eta(\gamma)(a_{0}) \cdot e_{0} + \cdots + \eta(\gamma)(a_{n}) \cdot e_{n}. \]

\begin{lemma}
  For all $s \in S$, one has $\eta(\gamma)(s) = s$.
\end{lemma}
\begin{proof}
  Let $0 < i \le n$ be given.
  Since $s \cdot e_{0} - e_{i}$ represents $\pfrak(\Hfrak_{x_{i}-s})$, Proposition~\ref{proposition:rigidification} again shows that one has $\Gamma(s \cdot e_{0} - e_{i}) = \epsilon \cdot (s \cdot e_{0} - e_{i})$, for some $\epsilon \in k^{\times}$.
  On the other hand, we have $\Gamma(s \cdot e_{0} - e_{i}) = \eta(\gamma)(s) \cdot e_{0} - e_{i}$.
  Since $e_{0}$ and $e_{i}$ are linearly independent, we find $\epsilon = 1$ and $\eta(\gamma)(s) = s$, as claimed.
\end{proof}

As $S$ generates $k_{0}$, the image of $\eta : \Gfrak \to \Aut(k)$ lands in $\Gal(k|k_{0})$.
Note that when $\gamma$ is the image of $\sigma \in \Gal(k|k_{0})$, one has $\gamma \Hfrak_{x_{i} - c} = \Hfrak_{x_{i}-\sigma c}$ for all $i = 1,\ldots,n$ and all $c \in k$.
Unfolding the definitions, this in turn implies that, in this case, one has $\eta(\gamma) = \sigma$.
In other words, the morphism $\eta : \Gfrak \to \Gal(k|k_{0})$ just constructed is a section to the canonical map $\rho : \Gal(k|k_{0}) \to \Gfrak$.
Since it will come up in the following section, we note that for all $(a_{0}:\cdots:a_{n}) \in \Pbb^{n}(k)$ and $\gamma \in \Gfrak$ with image $\sigma \in \Gal(k|k_{0})$, that one has
\[ \gamma (a_{0}:\cdots : a_{n}) = (\sigma a_{0} : \cdots : \sigma a_{n}). \]

\subsection{Conclusion}

To conclude, we prove that any element in the kernel of $\eta : \Gfrak \to \Gal(k|k_{0})$ is contained in the image of the canonical map $\Lambda^{\times} \to \Aut(\Pi^{a}_{\Hscr_{\xbf,S}})$.
Suppose $\gamma$ is in the kernel of $\eta$.
By the observation from the last subsection, it follows that $\gamma$ fixes every point of $\Pbb^{n}(k)$, hence $\gamma$ fixes every hyperplane $\Hfrak$, and hence
\[ \gamma \Ibf_{\Hfrak} = \Ibf_{\Hfrak}, \ \gamma \Dbf_{\Hfrak} = \Dbf_{\Hfrak} \]
for all hyperplanes $\Hfrak$.
As $\Ibf_{\Hfrak} \cong \Lambda$, it follows that there exists some $\epsilon_{\Hfrak} \in \Lambda^{\times}$ such that $\gamma$ acts as multiplication by $\epsilon_{\Hfrak}$ on $\Ibf_{\Hfrak}$.
Since $\Lambda \cdot \gamma [\Hfrak] = \Lambda \cdot [\Hfrak]$ as well, we find that there exists $\epsilon'_{\Hfrak}$ such that $\gamma [\Hfrak] = \epsilon'_{\Hfrak} \cdot [\Hfrak]$.
The pairing~\eqref{eqn:univ-H1-pairing} induces a duality between $\Ibf_{\Lfrak}$ and $\HHbf^{1}_{\Afrak}/\Ubf_{\Hfrak}$ which is compatible with the action of $\gamma$.
As $\HHbf^{1}_{\Afrak}/\Ubf_{\Hfrak}$ is generated by $[\Hfrak]$, we deduce that $\epsilon_{\Hfrak}' = \epsilon_{\Hfrak}^{-1}$.

\begin{lemma}\label{lemma:epsilon-indep}
  Suppose that $(\Hfrak_{1},\Hfrak_{2},\Hfrak_{3})$ is a dependent triple.
  Then $\epsilon_{\Hfrak_{1}} = \epsilon_{\Hfrak_{2}} = \epsilon_{\Hfrak_{3}}$.
\end{lemma}
\begin{proof}
  Put $\epsilon_{i} := \epsilon_{\Hfrak_{i}}$.
  Let $\bar \sigma \in \Pibf^{a}_{\Hfrak_{1}}$ be an element satisfying $\bar\sigma[\Hfrak_{1} \cap \Hfrak_{2}] = 1$, and let $\sigma \in \Dbf_{\Hfrak_{1}}$ be an element mapping to $\bar\sigma$ under the canonical projection $\Dbf_{\Hfrak_{1}} \twoheadrightarrow \Pibf^{a}_{\Hfrak_{1}}$.
  Hence $\sigma[\Hfrak_{2}] = 1$ and thus $\sigma[\Hfrak_{3}] = 1$ as well.
  Note that $\gamma \sigma$ is an element of $\Dbf_{\Hfrak_{1}}$ as well, since $\gamma \Hfrak_{1} = \Hfrak_{1}$.
  Hence:
  \[ (\gamma \sigma)[\Hfrak_{2}] = \sigma(\gamma^{-1}[\Hfrak_{2}]) = \epsilon_{2} \cdot \sigma [\Hfrak_{2}] = \epsilon_{2}, \]
  and similarly $(\gamma \sigma)[\Hfrak_{3}] = \epsilon_{3}$.
  But since $\gamma \sigma \in \Dbf_{\Hfrak_{1}}$, it follows that $(\gamma \sigma)[\Hfrak_{2}] = (\gamma \sigma)[\Hfrak_{3}]$, hence $\epsilon_{2} = \epsilon_{3}$.
  Repeating the argument with $\Hfrak_{1},\Hfrak_{2}$ interchanged, we deduce the claim.
\end{proof}

By Lemma~\ref{lemma:epsilon-indep}, we deduce that $\epsilon_{\Hfrak}$ does not depend on $\Hfrak$.
Indeed, if $\Hfrak_{1},\Hfrak_{2}$ meet, then they are part of a dependent triple hence $\epsilon_{\Hfrak_{1}} = \epsilon_{\Hfrak_{2}}$.
If $\Hfrak_{1},\Hfrak_{2}$ do not meet, then we find a third $\Hfrak_{3}$ which meets both to again deduce that $\epsilon_{\Hfrak_{1}} = \epsilon_{\Hfrak_{3}} = \epsilon_{\Hfrak_{2}}$.
As $\Pibf^{a}_{\Afrak}$ is generated by the $\Ibf_{\Hfrak}$ as $\Hfrak \in \Planes$ varies, it follows that there exists an $\epsilon$ such that $\gamma$ acts on the whole $\Pibf^{a}_{\Afrak}$ as multiplication by $\epsilon$.
For every $\Ucal_{\Hbf}$ in $\Hscr_{\xbf,S}$, the canonical map
\[ \Pibf^{a}_{\Afrak} \to \Pi^{a}_{\Ucal_{\Hbf}} \]
is surjective and $\gamma$-equivariant, hence $\gamma$ acts as multiplication by $\epsilon$ on $\Pi^{a}_{\Ucal_{\Hbf}}$ as well.
And finally, since $\pi_{z} : \Pibf^{a}_{\Afrak} \to \Pi^{a}_{\Pbb^{1} \smin \{0,1,\infty\}}$ is surjective for any one of the $z \in \{x_{1},\ldots,x_{n}\}$, it follows again that $\gamma$ acts as multiplication by $\epsilon$ on $\Pi^{a}_{\Pbb^{1} \smin \{0,1,\infty\}}$.
In other words, $\gamma$ acts as multiplication by $\epsilon$ on $\Pi^{a}_{X}$ for any object $X$ of $\Hscr_{\xbf,S}$, which shows that, indeed, $\gamma$ is in the image of the canonical map $\Lambda^{\times} \to \Aut(\Pi^{a}_{\Hscr_{\xbf,S}})$.
This concludes the proof of Theorem~\ref{maintheorem:general}.

\bibliography{refs}

\end{document}